\documentclass[preprint]{imsart}

\RequirePackage{amsthm,amsmath,amsfonts,amssymb,amstext,comment}
\RequirePackage[numbers]{natbib}
\RequirePackage[colorlinks,citecolor=blue,urlcolor=blue]{hyperref}
\RequirePackage{graphicx}

\arxiv{2010.00000}
\startlocaldefs
\theoremstyle{plain}

\newtheorem{theorem}{Theorem}[section]

\allowdisplaybreaks
\theoremstyle{definition}

\theoremstyle{remark}

\DeclareMathOperator*{\argmin}{arg\,min}

\newcommand{\Omegabf}{{\boldsymbol \Omega}}
\newcommand{\Gammabf}{{\boldsymbol \Gamma}}
\newcommand{\Sigmabf}{{\boldsymbol \Sigma}}
\newcommand{\Psibf}{{\boldsymbol \Psi}}
\newcommand{\Deltabf}{{\boldsymbol \Delta}}
\newcommand{\Yibf}{{\mathbf{Y}}_{i}}
\newcommand{\Xibf}{{\mathbf{X}}_{i}}
\newcommand{\tXibf}{\tilde{\mathbf{X}}_{i}}
\newcommand{\Zibf}{{\mathbf{Z}}_{i}}
\newcommand{\Ibf}{{\mathbf{I}}}
\newcommand{\Qbf}{{\mathbf{Q}}}
\newcommand{\Abf}{{\mathbf{A}}}
\newcommand{\Bbf}{{\mathbf{B}}}
\newcommand{\Dbf}{{\mathbf{D}}}
\newcommand{\Gbf}{{\mathbf{G}}}
\newcommand{\Rbf}{{\mathbf{R}}}
\newcommand{\Sbf}{{\mathbf{S}}}
\newcommand{\Xbf}{{\mathbf{X}}}
\newcommand{\Ybf}{{\mathbf{Y}}}
\endlocaldefs

\begin{document}
\begin{frontmatter}
\title{High dimensional convergence rates for sparse precision estimators for matrix-variate data}
\runtitle{Sparse precision convergence rates for matrix-variate data}

\begin{aug}
\author[A]{\fnms{Hongqiang}~\snm{Sun}\ead[label=e1]{sunh1@ufl.edu}}
\and
\author[B]{\fnms{Kshitij}~\snm{Khare}\ead[label=e2]{kdkhare@ufl.edu}}

\address[A]{Department of Statistics,
University of Florida\printead[presep={,\ }]{e1}}

\address[B]{Department of Statistics,
University of Florida\printead[presep={,\ }]{e2}}
\runauthor{Sun and Khare}
\end{aug}

\begin{abstract}
In several applications—particularly in biostatistics and bioinformatics—the underlying structure of the data allows samples to be organized in a matrix‑variate form. For example, gene expression studies routinely measure thousands of genes across multiple tissues for each individual, producing gene‑by‑tissue matrices whose dependence structures along rows and columns carry critical biological meaning. Similar matrix‑valued formats arise in neuroscience, where multiregion or multitime‑point EEG measurements form time‑by‑location matrices characterized by rich covariance patterns across both dimensions. In such settings, the underlying row and column covariance matrices are fundamental quantities of interest. In this setting, we establish convergence rates for a popular penalized sparse estimator called SMGM in high-dimensional settings where the row and column dimensions are allowed to increase with the sample size. We discuss critical errors in the previous high-dimensional convergence rate analyses for the SMGM estimator in existing literature, and highlight the novel features of our approach. 
\end{abstract}



\end{frontmatter}

\section{Introduction}\label{intro}

\noindent
Many modern multivariate datasets have an inherent matrix structure, where each variable corresponds to a combination of levels of two factors. Prominent examples arise in biology and neuroscience: in genomics, the GTEx project measures gene expression for thousands of genes across dozens of human tissues, yielding natural gene--by--tissue matrices that motivate separate row (gene) and column (tissue) dependence modeling \citep{gtex2017nature,gtex2015science}; in electrophysiology, multichannel EEG recordings deliver rich time--by--channel matrices in which spatial (across sensors/channels) and temporal covariances play distinct roles, as exemplified by widely used datasets such as the PhysioNet Motor Movement/Imagery EEG collection \citep{physionet_eegmmidb}. Applications also abound in other fields, such as finance and economics. For example, the dataset in \cite[Section 5.1]{leng2012sparse} records annual US exports to 13 geographical regions (Factor~1) for 36 export items (Factor~2), and can naturally be arranged as a $13 \times 36$ matrix. 

In particular, consider a dataset with $n$ independent and identically distributed $p \times q$ matrices $\Ybf_1, \Ybf_2, \cdots, \Ybf_n$, with $\Ybf_1$ following a matrix variate normal distribution with mean $\bf{0}$ (a $p \times q$ matrix with all zero entries), with row covariance matrix $\Sigmabf$ and column covariance matrix $\Psibf$. This is equivalent to assuming that $\operatorname{vec}(\Ybf_1)$ (a vector obtained by stacking the columns of $\Ybf_1$) has a $pq$-variate normal distribution with mean ${\underline 0} \in \mathbb{R}^{pq}$
and covariance matrix $\Psibf \otimes \Sigmabf$ (see \cite{Dawid:1981, Gupta:Nagar:2000}). Hence, the above model essentially specifies a multivariate normal distribution for the vectorized version of the observed matrices, with the covariance matrix constrained to a Kronecker product form. This constraint is motivated/justified by the aforementioned structure among the variables. The estimation of $\Sigmabf$ and $\Psibf$ is of fundamental importance in these models, as it reveals the nature of dependence between the various variables. 

Note that the likelihood for $(\Sigmabf, \Psibf)$ is given by 
\begin{eqnarray} \label{likelihood}
L(\Sigmabf, \Psibf) 
&=& \frac{1}{\sqrt{2 \pi}^n |\Psibf \otimes \Sigmabf|^{n/2}} \exp \left( -\frac{1}{2} \sum_{i=1}^n \operatorname{vec} (\Yibf)^{\top} (\Psibf \otimes \Sigmabf)^{-1} \operatorname{vec}(\Yibf) \right) \nonumber\\
&=& \frac{1}{\sqrt{2 \pi}^n |\Psibf|^{np/2} |\Sigmabf|^{nq/2}} \exp \left( -\frac{1}{2} \sum_{i=1}^n \operatorname{tr}(\Sigmabf^{-1} \Yibf \Psibf^{-1} \Yibf^{\top}) \right). 
\end{eqnarray}

\noindent
It is evident that $L(\Sigmabf, \Psibf) = L(c \Sigmabf, c^{-1} \Psibf)$ for every 
$c > 0$ (\cite{Galecki:1994, Naik:Rao:2001}). In other words, the parameter $(\Sigmabf, \Psibf)$ is identifiable only up to a positive multiplicative constant. Hence, a one-dimensional constraint such as $\Psibf_{qq} = 1$ (\cite{SVV:2008}), or $\Sigmabf_{pp} = 1$ (\cite{Wang:West:2009}), or $\operatorname{tr}(\Psibf) = q$ (\cite{Theobald:Wuttke:2006}) is often imposed 
for identifiability. 

It is easy to see that conditional maximizers of the likelihood function $L$ with respect to $\Sigmabf$ (given $\Psibf$) or with respect to $\Psibf$ (given $\Sigmabf$) are available in closed form. This observation was used to develop an iterative alternating minimization approach, called the flip-flop algorithm for finding the maximum likelihood estimator of $(\Sigmabf, \Psibf)$, see \cite{Dutilleul:1999, Lu:Zimmerman:2005, SVV:2008}. A non-iterative three-step version of this estimator was also considered in \cite{THZ:2013, WJS:2008}. Asymptotic properties of the non-iterative flip-flop estimator were established in \cite{WJS:2008} (in the $p,q$ fixed and $n \rightarrow \infty$ setting) and \cite{ franks2021near, THZ:2013} (in the high-dimensional setting where $p,q$ are allowed to grow with $n$). 

Introducing sparsity in the inverse covariance matrix is a popular and effective device to tackle parameter proliferation in high-dimensional multivariate data settings \cite{Yuan:Lin:2007}. Under Gaussianity, zeros in the inverse covariance matrix correspond to conditional independence relationships between relevant variables \cite{Lauritzen:1996}. In the matrix-variate data setting, several methods for sparse estimation of $\Omegabf = \Sigmabf^{-1}$ and $\Gammabf = \Psibf^{-1}$ have been proposed \cite{leng2012sparse, THZ:2013, Yin:Li:2012,  Zhou:2014}. In \cite{leng2012sparse, Yin:Li:2012}, the authors consider an objective function which combines the negative log-likelihood with a sparsity-inducing penalty (for $\Omegabf$ and $\Gammabf$). In particular, if a lasso penalty is used for off-diagonal entries of $\Omegabf$ and $\Gammabf$, then the objective function is given by 
\begin{equation}\label{gobj}
 g (\Omegabf, \Gammabf) = \frac{1}{n p q} \sum_{i=1}^{n} \operatorname{tr}\left(\Yibf \Gammabf \Yibf^{\top} \Omegabf\right)-\frac{1}{p} \log |\Omegabf|-\frac{1}{q} \log |\Gammabf|+ \lambda_1 \sum_{i \neq j} |\Omega_{ij}| + \lambda_2 \sum_{i \neq j} |\Gamma_{ij}|.  
\end{equation}
Again, conditional minimizers of this objective function with respect to $\Omegabf$ (given $\Gammabf$) or with respect to $\Gammabf$ (given $\Omegabf$) can be obtained through relevant established approaches (such as the graphical lasso) in the multivariate setting. Both \cite{leng2012sparse} and \cite{Yin:Li:2012} leverage this observation to develop an iterative alternating minimization approach for the objective function in (\ref{gobj}). Following \cite{leng2012sparse}, we will refer to the resulting estimator of $(\Omegabf,\Gammabf)$ as the SMGM-lasso estimator, and the corresponding algorithm as the SMGM algorithm. Both papers provide high-dimensional convergence rates for the respective estimators. Convergence rates for a non-iterative three step version of the SMGM-lasso estimator (called KGlasso) were established in \cite{THZ:2013}. The Gemini method in \cite{Zhou:2014} uses a partial correlation based approach to construct a separate penalized objective function each for $\Omegabf$ and $\Gammabf$, and obtain sparse estimators by (independent/separate) iterative optimization of these two functions. High-dimensional consistency is established under appropriate regularity conditions even when the sample consists of only one (matrix-variate) observation. See also \cite{10.3150/17-BEJ980}, where a similar separate optimization based approach is undertaken to provide sparse estimators of the row and column covariance matrices $\Omegabf^{-1}$ and $\Gammabf^{-1}$. 


From a statistical efficiency point of view, the SMGM-lasso estimator in \cite{leng2012sparse, Yin:Li:2012} is attractive because it fully leverages the dependence in the data by {\it jointly} minimizing the penalized log-likelihood function with respect to $(\Omegabf, \Gammabf)$. 
The authors in \cite{leng2012sparse} and \cite{Yin:Li:2012} provide asymptotic convergence rates for the SMGM-lasso estimator under mild regularity assumptions (Theorem 1 in \cite{leng2012sparse}, and Theorem 3 in \cite{Yin:Li:2012}). Unfortunately there are critical errors in the proof of both results. The errors in both these arguments are described in detail in Section \ref{Error}. These errors are substantive in nature and cannot be rectified through simple modifications to the respective arguments. As the {\bf key contribution of this paper}, we establish high-dimensional convergence rates for the SMGM-lasso estimator by employing novel strategies, constructions and arguments. In Section \ref{Error}, we compare and contrast our innovative approach with those presented in \cite{leng2012sparse, Yin:Li:2012}. 

In \cite{leng2012sparse}, consistency results are also provided for other sparsity-inducing penalties such as SCAD. However, the errors previously mentioned are only related to the log-likelihood component of the objective function 
$g$ and do not affect the penalty component. Therefore, for simplicity and clarity, we will focus on the lasso penalty setting throughout this paper.

We now describe the {second class of estimators that are analyzed in this paper}. Note that in the context of covariance estimation for $d$-dimensional vector-variate data using $n$ i.i.d. observations (with $d >> n$), consistent estimation of the covariance matrix is not possible unless a low-dimensional structure such as sparsity is imposed. However, for estimating $(\Sigmabf, \Psibf)$ in the current matrix-variate context, it appears that such low-dimensional structures may not always be necessary for achieving consistency. The reason is that we have $nq$ (dependent) observations with covariance matrix $\Sigmabf$ and $np$ dependent observations with covariance matrix $\Psibf$. Hence, as long as $p = o(nq)$ and $q = o(np)$, or equivalently $\max \left( \frac{p}{q}, \frac{q}{p} \right) = o(n)$, one could expect consistent estimation under mild regularity conditions which do not impose any low-dimensional structure. With these ideas in mind, we examine the `heuristic' sample covariance estimators proposed by Srivastava et al. \cite{SVV:2008}. In the high-dimensional context, straightforward adaptations to the analysis of the `sample correlation estimators' in 
\cite{Zhou:2014}, can be used to obtain non-asymptotic high probability bounds for the entry-wise maximum differences between these heuristic estimators and the corresponding true parameter values. Leveraging these results together with standard matrix norm inequalities leads to a spectral norm consistency for the heuristic estimators under the constraint $\max \left( \frac{p^2\log(\max (p,q))}{q}, \frac{q^2\log(\max (p,q))}{p} \right) = o(n)$, which is much more restrictive than expected. As {\bf an additional contribution of the paper}, we consider a high-dimensional setting where $p,q$ are allowed to grow with $n$ but with the much milder restriction $ \frac{\max (p,\log n)}{q}=o(n)$ and $ \frac{\max (q,\log n)}{p}=o(n)$. In this setting, we show that the heuristic estimators are consistent in spectral norm. More importantly, we obtain asymptotic high-dimensional spectral norm convergence rates for both these heuristic estimators (see Theorem \ref{HEmain}). 

The remainder of the paper is organized as follows. High-dimensional Frobenius norm convergence rates for the SMGM estimator are established in 
Section \ref{high:dim:consistency}. A detailed description of the errors/issues in previous consistency proofs for the SMGM estimator is provided in Section \ref{Error}. Finally, high-dimensional spectral norm convergence rates for the heuristic estimator are established in 
Section \ref{convergence:rate:heuristic:estimator}. 



\section{High-dimensional convergence rates for the Sparse SMGM Estimator} 
\label{high:dim:consistency}

\noindent
In this section, we will establish high-dimensional convergence rates 
for the penalized sparse estimator (SMGM) proposed by Leng \& Tang \cite{leng2012sparse}. As discussed in the introduction, we use several novel strategies and arguments which help us avoid the pitfalls/errors in the arguments of \cite{leng2012sparse, Yin:Li:2012}. The different arguments are compared and contrasted in Section \ref{Error}. 

We start by specifying the true data generating model. Under this model, for each $n$, the random matrices $\Ybf_{1}, \Ybf_{2}, \cdots, \Ybf_{n}$ are independent and identically distributed with a matrix normal distribution, which has mean $\mathbf{0}$, row covariance matrix $\mathbf{\Sigmabf}_0$ and column covariance matrix $\mathbf{\Psibf}_0$. Let $\Omegabf_{0}=\Sigmabf_{0}^{-1}$ 
and $\Gammabf_{0}=\Psibf_{0}^{-1}$ respectively denote the row and column precision matrix, and $P_0$ denote the probability measure underlying the true data generating model. Note that since $p$ and $q$ depend on $n$, the covariance matrices, precision matrices and $\Yibf$'s all depend on $n$. However, we omit their dependence on $n$ for simplicity of notation.

Let $S_1 = \{(i, j) : (\Omega_{0})_{ij} \neq 0\}$ and
$S_2 = \{(i, j) : (\Gamma_{0})_{ij} \neq 0\}$ denote the locations of non-zero entries in $\Omegabf_0$ and $\Gammabf_0$ respectively. Let $s_1 = |S_1|-p$ and $s_2 = |S_2|-q$ be the number of nonzero off-diagonal parameters in $\Omegabf_0$ and $\Gammabf_0$ respectively. Under this setup, we want to study the asymptotic properties of the SMGM estimators $\hat{\Omegabf}$ and $\hat{\Gammabf}$ defined by 
\begin{equation}\label{g}
 (\hat{\Omegabf}, \hat{\Gammabf}) = \argmin_{\Omegabf \succ {\bf 0}, \Gammabf \succ {\bf 0}} g (\Omegabf, \Gammabf). 
 \end{equation}

\noindent
Here, according to the the Loewner order, we say that ${\bf A} \succ {\bf B}$ if $\Abf-\Bbf$ is positive definite.  Similarly, we say that $\Abf \succcurlyeq \Bbf$ if $\Abf-\Bbf$ is positive semi-definite. In order to establish our asymptotic results, we need the following mild
regularity assumptions. Each assumption below is followed by an interpretation/discussion. Assumptions 1,2 and 4 are identical to the relevant assumptions in \cite{leng2012sparse}. Assumption 3 is a slightly stronger version of \cite[Assumption 3]{leng2012sparse} (with $1+p /\left(s_{1}+1\right)$ and $1+q /\left(s_{2}+1\right)$ in \cite{leng2012sparse} replaced by $\sqrt{1+p /\left(s_{1}+1\right)}$ and $\sqrt{1+q /\left(s_{2}+1\right)}$ respectively here). Assumption 5 is an additional mild assumption that is needed to bound a key term in our analysis, see 
Section \ref{T31} for more details. We would like to clarify that this additional assumption by itself is not enough to fix the errors in the proofs of \cite{leng2012sparse, Yin:Li:2012}. The novel constructions and techniques that are used in our arguments below play a critical and indispensable role, see Section \ref{T32} in particular for more details.  

\medskip

\noindent
\textit{Assumption} 1. As $n \rightarrow \infty, \left(p+s_{1}\right) \log p /(n q) {\longrightarrow 0}$ and $\left(q+s_{2}\right) \log q /(n p) {\longrightarrow} 0$.

\smallskip

\noindent
The conditions in this assumption restrict the density of the true precision matrices $\Omegabf_0$ and 
$\Gammabf_0$. Similar assumptions are common in high-dimensional penalized sparse estimation for precision matrices; see, for example, \cite{leng2012sparse,rothman2008sparse}. 

\medskip

\noindent
\textit{Assumption}  2. There exists constant $\tau_{1}>0$ such that for all $n \geqslant 1$,
\begin{eqnarray*}
&& 0<\tau_{1}<\nu_{1}\left(\Sigmabf_{0}\right) \leqslant \nu_{p}\left(\Sigmabf_{0}\right)<1 / \tau_{1}<\infty , \\
&& 0<\tau_{1}<\nu_{1}\left(\Psibf_{0}\right) \leqslant \nu_{q}\left(\Psibf_{0}\right)<1 / \tau_{1}<\infty .
\end{eqnarray*}
Here $\nu_{1}(\Abf) \leqslant \nu_{2}(\Abf) \leqslant \cdots \leqslant \nu_{m}(\Abf)$ denote the 
eigenvalues of an $m$-dimensional symmetric matrix $\Abf$.
\smallskip

\noindent
This assumption essentially states that the 
eigenvalues of the variance matrices $\Sigmabf_{0}$ and $\Psibf_{0}$ should be (uniformly in $n$) 
bounded away from zero and infinity. This is a very standard assumption in high-dimensional 
covariance asymptotics; see, for example, \cite{banerjee2014posterior,banerjee2015bayesian,khare2015convex,leng2012sparse,peng2009partial}. 

\medskip

\noindent
\textit{Assumption}  3. The tuning parameters $\lambda_1$ and $\lambda_2$ satisfy
\begin{eqnarray*}
&&\lambda_{1}=O\left[p^{-1}\sqrt{1+p /\left(s_{1}+1\right)} \sqrt{\log p /(n q)}\right]\\
&&\lambda_{2}=O\left[q^{-1}\sqrt{1+q /\left(s_{2}+1\right)} \sqrt{\log q /(n p)}\right].
\end{eqnarray*}

\medskip

\noindent
\textit{Assumption}  4. As $n \rightarrow \infty$, the tuning parameters $\lambda_1$ and $\lambda_2$ satisfy
\begin{eqnarray*}
\lambda_{1}^{-2} p^{-2}\log p /(n q) \rightarrow 0; \lambda_{2}^{-2} q^{-2} \log q /(n p) \rightarrow 0.
\end{eqnarray*}

\smallskip

\noindent
While Assumption 3 sets upper bounds for tuning parameters $\lambda_1$ and $\lambda_2$, 
Assumption 4 sets lower bounds for them. In other words, a delicate balance needs to be struck in 
the amount of penalization for consistent estimation in this challenging high-dimensional setting.
Note that the collection of settings in which Assumptions 3 \& 4 are simultaneously satisfied is by no means vacuous or trivial. For example, both assumptions are satisfied when $s_1=o(p)$, $s_2=o(q)$, $\lambda_1=C_{0}\sqrt{\frac{\log p}{npq(s_1+1)}}$ and $\lambda_2=C_{0}\sqrt{\frac{\log q}{npq(s_2+1)}}$ for some constant $C_{0}$.
Our final mild assumption adds extra restriction on relation between $n,p,q$ and sparsity compared to Leng and Tang's assumptions in \cite{leng2012sparse}.   

\medskip

\noindent
\textit{Assumption} 5. Let 
\begin{eqnarray*}
r_n &:=&\max \left(1, \frac{s_{2}}{q}, \frac{s_{1} q}{p},  \frac{\left(q+s_{2}\right) \log q}{\lambda_{1}^{2}n p^{3}},  \frac{\left(p+s_{1}\right) \log p}{\lambda_{2}^{2}n p q^{3}}\right) \text{ and}\\
r_n' &:=&\max \left(1, \frac{s_{1}}{p}, \frac{s_{2} p}{q}, \frac{\left(q+s_{2}\right) \log q}{\lambda_{1}^{2}n q p^{3}}, \frac{\left(p+s_{1}\right) \log p}{\lambda_{2}^{2}n q^{3}}\right).
\end{eqnarray*}

\noindent
Then $ \frac{r_n\log p q}{n} \longrightarrow 0$ or $ \frac{r_n'\log p q}{n} \longrightarrow 0$ as $ n \rightarrow \infty$. 

\smallskip

\noindent
When $p \geqslant q$, this assumption would be satisfied, for instance, by selecting the aforementioned  $\lambda_1$ and $\lambda_2$ while ensuring that $s_1=o(p)$, $s_2=o(q)$  and $\max \left( s_{1}+1, \frac{p (s_{2}+1) \log p}{q^{2} \log q}\right)\frac{\log p q}{n} \longrightarrow 0$ hold. On the other hand, when $q \geqslant p$, this assumption would be satisfied, for instance, by selecting  the aforementioned  $\lambda_1$ and $\lambda_2$ while ensuring that $s_1=o(p)$, $s_2=o(q)$  and\\
$\max \left( s_{2}+1, \frac{q (s_{1}+1) \log q}{p^{2} \log p}\right)\frac{\log p q}{n}
 \longrightarrow 0$ hold.

\medskip

\noindent
With the required assumptions in hand, we now state our main consistency result. 
\begin{theorem} \label{Fnormcr}
(Frobenius norm convergence rates for SMGM estimator) Under Assumptions 1-5,
there exists a local minimizer $(\hat{\Omegabf}, \hat{\Gammabf})$ of (\ref{g}) such that
$$
\frac{\left\|\hat{\Omegabf}-\Omegabf_0\right\|_{\mathrm{F}}^2}{p}=O_{P_0}\left\{\left(1+\frac{s_1}{p}\right) \log p /(n q)\right\}
$$
and
$$
\frac{\left\|\hat{\Gammabf}-\Gammabf_0\right\|_{\mathrm{F}}^2}{q}=O_{P_0}\left\{\left(1+\frac{s_2}{q}\right) \log q /(n p)\right\} .
$$
\end{theorem}
\begin{proof}
We start by establishing required notation for the proof. 
\begin{eqnarray*}
&&\Deltabf_1=\Omegabf-\Omegabf_0,\Deltabf_2=\Gammabf-\Gammabf_0\\
&&
\tilde{\Deltabf}_{1}=\Omegabf_{0}^{-1 / 2} \Deltabf_{1} \Omegabf_{0}^{-1 / 2}, \tilde{\Deltabf}_{2}=\Gammabf_{0}^{-1 / 2} \Deltabf_{2} \Gammabf_{0}^{-1 / 2}\\
&&\zeta_1=\max(s_1,p),\zeta_2=\max(s_2,q)\\
&& \alpha_{1}=\left\{\zeta_{1} \log p /(n q)\right\}^{1 / 2}, \beta_{1}=\{p \log p /(n q)\}^{1 / 2} \\
&& \alpha_{2}=\left\{\zeta_{2} \log q /(n p)\right\}^{1 / 2}, \beta_{2}=\{q \log q /(n p)\}^{1 / 2}
\end{eqnarray*}
\noindent
For any positive integer $m$, let $\mathcal{D}_{m}$ denote the space of $m \times m$ diagonal matrices with positive diagonal entries, and let $\mathcal{R}_m$ denote the space of $m \times m$ matrices with all diagonal entries equal to zero. For a given positive constant $C$ (the choice of $C$ will be determined later in the proof), define the spaces $\mathcal{A}_p$ and $\mathcal{B}_p$ as 
{\small
\begin{eqnarray*}
\mathcal{A}_p &:=& \left\{\mathbf{M}: \; \mathbf{M}=\alpha_{1} \Rbf_{p}+\beta_{1} \Dbf_{p}, \; 
\Dbf_p \in \mathcal{D}_p, \; \Rbf_p \in \mathcal{R}_p \mbox{ and } \|\Dbf_p\|_F = \|\Rbf_p\|_F = C \right\}, \\
\mathcal{B}_q &:=& \left\{\mathbf{M}: \; \mathbf{M}=\alpha_{2} \Rbf_{q}+\beta_{2} \Dbf_{q}, \; 
\Dbf_q \in \mathcal{D}_q, \; \Rbf_q \in \mathcal{R}_q \mbox{ and } \|\Dbf_q\|_F = \|\Rbf_q\|_F = C \right\}. 
\end{eqnarray*}}

\noindent
We also establish some error bounds between relevant `sample' covariance matrices and their `population' counterparts. These bounds will be useful for the subsequent analysis. Let $\Xibf=\Yibf \Gammabf_{0}^{1 / 2}$ for $1 \leqslant i \leqslant n$. It follows that ${\bf X}_1, {\bf X}_2, \cdots, {\bf X}_n$ are i.i.d. from a matrix-variate normal distribution with mean matrix ${\bf 0}$, row covariance matrix $\Sigmabf_0$, and column covariance matrix ${\bf I}_q$. Hence, the matrix 
$$
\Qbf_{1}=(n q)^{-1} \sum_{i=1}^{n} \Xibf \Xibf^{\top} 
$$

\noindent
can be viewed as a sample covariance matrix obtained from $nq$ i.i.d. observations from $p$-variate normal distribution with mean vector $\underline{0}$ and covariance matrix $\Sigmabf_0$. Using \cite[Lemma A.3]{bickel2008regularized}, there exists a constant $K_{1}$ such that
\begin{equation} \label{Q1}
P_{0}\left\{\max _{1 \leqslant i, j \leqslant p}\left|\left(Q_{1}\right)_{i j}-\left(\Sigma_{0}\right)_{i j}\right| \leqslant K_{1} \sqrt{\frac{\log p}{n q}}\right\}\rightarrow 1. 
\end{equation}
Let $\Zibf=\Omegabf_{0}^{1 / 2}\Yibf$ for $1 \leqslant i \leqslant n$. It follows that ${\bf Z}_1, {\bf Z}_2, \cdots, {\bf Z}_n$ are i.i.d. from a matrix-variate normal distribution with mean matrix ${\bf 0}$, row covariance matrix ${\bf I}_q$, and column covariance matrix $\Psibf_0$. Hence, the matrix 
$$
\Qbf_{2}=(n p)^{-1} \sum_{i=1}^{n} \Zibf^{\top} \Zibf
$$

\noindent
can be viewed as a sample covariance matrix obtained from $np$ i.i.d. observations from $q$-variate normal distribution with mean vector $\underline{0}$ and covariance matrix $\Psibf_0$. Using \cite[Lemma A.3]{bickel2008regularized}, there exists a constant $K_{2}$ such that
\begin{equation}\label{Q2}
P_{0}\left\{\max _{1 \leqslant i, j \leqslant q}\left|\left(Q_{2}\right)_{i j}-\left(\Psi_{0}\right)_{i j}\right| \leqslant K_{2} \sqrt{\frac{\log q}{n p}}\right\}\rightarrow 1.\\
\end{equation}

\noindent
Finally, let 
$$
{\bf S} = \frac{1}{n} \sum_{i=1}^{n} \operatorname{vec}\left(\Yibf^{\top}\right) \operatorname{vec}\left(\Yibf^{\top}\right)^{\top} 
$$

\noindent
denote the $pq$-dimensional sample covariance matrix based on $n$ i.i.d. observations from a $pq$-variate normal distribution with mean vector $\underline{0}$ and covariance matrix $\Sigmabf_{0} \otimes \Psibf_{0}$. Again, by \cite[Lemma A.3]{bickel2008regularized}, there exists a constant $K_{3}$ such that 
\begin{equation} \label{Sk}
P_{0} \left\{\max _{1 \leqslant r, s \leqslant pq}\left|\left(S\right)_{rs}-\left(\Sigma_{0}\otimes\Psi_{0}\right)_{rs}\right| \leqslant K_{3} \sqrt{\frac{\log pq}{n }}\right\}  \rightarrow 1.
\end{equation}

\noindent
Let $K = \max(K_1, K_2, K_3)$ and define the sequences of events $\{C_{1,n}\}_{n \geqslant 1},\\\{C_{2,n}\}_{n \geqslant 1},
\{C_{3,n}\}_{n \geqslant 1}$ as 
\begin{eqnarray*}
 C_{1,n}&:=& \left\{\max _{1 \leqslant i, j \leqslant p}\left|\left(Q_{1}\right)_{i j}-\left(\Sigma_{0}\right)_{i j}\right| \leqslant K \sqrt{\frac{\log p}{n q}} \right\},\\
  C_{2,n}&:=&\left\{\max _{1 \leqslant i, j \leqslant q}\left|\left(Q_{2}\right)_{i j}-\left(\Psi_{0}\right)_{i j}\right| \leqslant K \sqrt{\frac{\log q}{n p}}\right\},\\
   C_{3,n}&:=&\left\{\max _{1 \leqslant r, s \leqslant pq}\left|\left(S\right)_{rs}-\left(\Sigma_{0}\otimes\Psi_{0}\right)_{rs}\right| \leqslant K \sqrt{\frac{\log pq}{n }}\right\},
\end{eqnarray*}
and let $C_n := C_{1,n}\cap C_{2,n} \cap C_{3,n}$. It follows from (\ref{Q1}), (\ref{Q2}) and (\ref{Sk}) that 
\begin{eqnarray}
\label{PCn}
   P_0 (C_n) \rightarrow 1 \mbox{ as } n \rightarrow \infty. 
\end{eqnarray}

\noindent
Our goal is to show that 
\begin{equation} \label{identity}
P_0 \left( \inf _{\tilde{\Deltabf}_{1} \in \mathcal{A}_p, \tilde{\Deltabf}_{2} \in \mathcal{B}_q}\left\{g\left(\Omegabf_{0}+\Deltabf_{1}, \Gammabf_{0}+\Deltabf_{2}\right)-g\left(\Omegabf_{0}, \Gammabf_{0}\right)\right\}>0  \right) \rightarrow 1 
\end{equation}

\noindent
as $n \rightarrow \infty$. Along with the definitions of $\mathcal{A}_p$ and $\mathcal{B}_q$, this would imply the existence of a local minimizer of $g$, say $(\hat{\Omegabf}, \hat{\Gammabf})$ that satisfies 
$$
\frac{\left\|\Omegabf_0^{-1/2} (\hat{\Omegabf}-\Omegabf_0)  \Omegabf_0^{-1/2} \right\|_{\mathrm{F}}^2}{p} = O_{P_0}\left\{\left(1+\frac{s_1}{p}\right) \log p /(n q)\right\}
$$
and
$$
\frac{\left\|\Gammabf_0^{-1/2} (\hat{\Gammabf}-\Gammabf_0) \Gammabf_0^{-1/2} \right\|_{\mathrm{F}}^2}{q} = O_{P_0}\left\{\left(1+\frac{s_2}{q}\right) \log q /(n p)\right\} .
$$

\noindent
Since $\|{\Abf^{1/2}} \Deltabf {\Abf^{1/2}}\|_F \leqslant \nu_{m}\left(\Abf\right) \|\Deltabf\|_F$ 
for any $m$-dimensional symmetric positive definite matrix ${\bf A}$ and any symmetric matrix $\Deltabf$, it follows by Assumption 2 that 
\begin{eqnarray}
\notag
\frac{\left\|\hat{\Omegabf}-\Omegabf_0\right\|_{\mathrm{F}}^2}{p}&\leqslant& \frac{1}{p\tau_{1}}\left\|\Omegabf_0^{-1/2} (\hat{\Omegabf}-\Omegabf_0)  \Omegabf_0^{-1/2} \right\|_{\mathrm{F}}^2 =O_{P_0}\left\{\left(1+\frac{\zeta_1}{p}\right) \log p /(n q)\right\}\\
\label{coveq1}
&=&O_{P_0}\left\{\left(1+\frac{s_1}{p}\right) \log p /(n q)\right\}
\end{eqnarray}
and
\begin{eqnarray}
\notag
\frac{\left\|\hat{\Gammabf}-\Gammabf_0\right\|_{\mathrm{F}}^2}{q}&\leqslant& \frac{1}{q\tau_{1}} \left\|\Gammabf_0^{-1/2} (\hat{\Gammabf}-\Gammabf_0) \Gammabf_0^{-1/2} \right\|_{\mathrm{F}}^2 = O_{P_0}\left\{\left(1+\frac{\zeta_2}{q}\right) \log q /(n p)\right\}\\
\label{coveq2}
&=& O_{P_0}\left\{\left(1+\frac{s_2}{q}\right) \log q /(n p)\right\}, 
\end{eqnarray}

\noindent
as required. Note equalities (\ref{coveq1}) and (\ref{coveq2}) follow from the fact that $\zeta_1=\max(s_1,p)$ and $\zeta_2=\max(s_2,q)$.  Hence, the focus in the subsequent analysis is to establish 
(\ref{identity}).

We first decompose $g(\Omegabf, \Gammabf)-g\left(\Omegabf_{0}, \Gammabf_{0}\right)$, and take a \textbf{different} path than in \cite{leng2012sparse} to bound the decomposed terms. Note by Taylor's expansion that 
$$
\begin{aligned}
\log |\Omegabf|-\log \left|\Omegabf_{0}\right|= & \operatorname{tr}\left(\Sigmabf_{0} \Deltabf_{1}\right)-\operatorname{vec}\left(\Deltabf_{1}\right)^{\top}  \\
& \times\left\{\int_{0}^{1} \left(\Omegabf_{w}^{-1}\otimes \Omegabf_{w}^{-1}\right)(1-w) d w\right\} \operatorname{vec}\left(\Deltabf_{1}\right)
\end{aligned}
$$
where $\Omegabf_{w}=\Omegabf_{0}+w \Deltabf_{1}$ and $
 \Omegabf_{w}^{-1} \otimes \Omegabf_{w}^{-1}=\left(\Omegabf_{0}+w \Deltabf_{1}\right)^{-1} \otimes\left(\Omegabf_{0}+w \Deltabf_{1}\right)^{-1} $. A similar expansion can be obtained for 
 $\log |\Gammabf| - \log |\Gammabf_0|$ with $\Gammabf_{w}=\Gammabf_{0}+w \Deltabf_{2}$ and $
 \Gammabf_{w}^{-1} \otimes \Gammabf_{w}^{-1}=\left(\Gammabf_{0}+w \Deltabf_{2}\right)^{-1} \otimes\left(\Gammabf_{0}+w \Deltabf_{2}\right)^{-1} $. Using these expansions along with the definition of $g$ in (\ref{gobj}), we obtain the decomposition 
$$
g(\Omegabf, \Gammabf)-g\left(\Omegabf_{0}, \Gammabf_{0}\right)=T_{1}+T_{2}+T_{3}+T_{4}+T_{5}+T_{6}+T_{7},
$$

\noindent
where

\begin{eqnarray*}
T_{1}&= & \frac{1}{n p q}\left\{\sum_{i=1}^{n} \operatorname{tr}\left(\mathbf{Y}_{i} \Gammabf_{0} \mathbf{Y}_{i}^{\mathrm{T}} \Deltabf_{1}\right)\right\}-\frac{\operatorname{tr}\left(\Sigmabf_{0} \Deltabf_{1}\right)}{p}, \\
T_{2}&= & \frac{1}{n p q}\left\{\sum_{i=1}^{n} \operatorname{tr}\left(\mathbf{Y}_{i}^{\mathrm{T}} \Omegabf_{0} \mathbf{Y}_{i} \Deltabf_{2}\right)\right\}-\frac{\operatorname{tr}\left(\Psibf_{0} \Deltabf_{2}\right)}{q}, \\
T_{3}&= & \frac{1}{n p q} \sum_{i=1}^{n} \operatorname{tr}\left(\mathbf{Y}_{i} \Deltabf_{2} \mathbf{Y}_{i}^{\mathrm{T}} \Deltabf_{1}\right), \\
T_{4}&= & p^{-1} \operatorname{vec}^{\mathrm{T}}\left(\Deltabf_{1}\right)\left\{\int_{0}^{1} \left(\Omegabf_{w}^{-1}\otimes \Omegabf_{w}^{-1}\right)(1-w) d w\right\} \operatorname{vec}\left(\Deltabf_{1}\right), \\
T_{5}&= & q^{-1} \operatorname{vec}^{\mathrm{T}}\left(\Deltabf_{2}\right)\left\{\int_{0}^{1} \left(\Gammabf_{w}^{-1}\otimes \Gammabf_{w}^{-1}\right)(1-w) d w\right\} \operatorname{vec}\left(\Deltabf_{2}\right), \\
T_{6}&=&\sum_{(i, j) \in S_{1}^{c}} \left\{\lambda_{1}\left|\Omega_{ij}\right|-\lambda_{1}\left|\left(\Omega_{0}\right)_{i j}\right|\right\}+\sum_{(i, j) \in S_{2}^{c}} \left\{\lambda_{2}\left|\Gamma_{ij}\right|-\lambda_{2}\left|\left(\Gamma_{0}\right)_{i j}\right|\right\}\\
&=&\sum_{(i, j) \in S_{1}^{c}} \lambda_{1}\left|\left(\Delta_{1}\right)_{i j}\right| +\sum_{(i, j) \in S_{2}^{c}} \lambda_{2}\left|\left(\Delta_{2}\right)_{i j}\right| 
, \text { and } \\
T_{7}&=&\sum_{(i, j) \in S_{1}, i \neq j} \left\{\lambda_{1}\left|\Omega_{ij}\right|-\lambda_{1}\left|\left(\Omega_{0}\right)_{i j}\right|\right\}+\sum_{(i, j) \in S_{2},i \neq j} \left\{\lambda_{2}\left|\Gamma_{ij}\right|-\lambda_{2}\left|\left(\Gamma_{0}\right)_{i j}\right|\right\}. 
\end{eqnarray*}

\noindent
The main steps of the remaining proof are as follows. We first find positive lower bounds for both $T_{4}$ and $ T_{5}$. Note also that $T_{6}$ is non-negative. Next, we carefully derive upper bounds for the remaining terms, $T_1, T_2, T_3, T_7$, and show that the infimum of $T_{4} + T_{5} + T_{6} - |T_{1}| - |T_2| - |T_{3}| - |T_{7}|$ on the set $\tilde{\Deltabf}_{1} \in \mathcal{A}_p, \tilde{\Deltabf}_{2} \in \mathcal{B}_q$ is strictly positive with high probability.


To find a lower bound for $T_4$, we need to find a lower bound (as a positive definite matrix in the Loewner order) for the term  $\Omegabf_{w}^{-1}\otimes \Omegabf_{w}^{-1}$ in the integrand. Note that 
{\small
\begin{eqnarray}
\Omegabf_{w}^{-1} \otimes \Omegabf_{w}^{-1}
&=&\left(\Omegabf_{0}+w \Deltabf_{1}\right)^{-1} \otimes\left(\Omegabf_{0}+w \Deltabf_{1}\right)^{-1} \nonumber\\
&=&(\Omegabf_{0}^{-1 / 2}\left(\Ibf+w \tilde{\Deltabf}_{1}\right)^{-1} \Omegabf_{0}^{-1 / 2}) \otimes (\Omegabf_{0}^{-1 / 2}\left(\Ibf+w \tilde{\Deltabf}_{1}\right)^{-1} \Omegabf_{0}^{-1 / 2}) \label{T4eq1}
\end{eqnarray}
}
\noindent
and 
{\small
\begin{equation} \label{T4eq2}
\nu_{\text{max}}\left(\Ibf+w \tilde{\Deltabf}_{1}\right)\leqslant \max_{\|\underline{x}\|_2=1} \underline{x}^{\top} \left(\Ibf+w \tilde{\Deltabf}_{1}\right) \underline{x}\leqslant \max_{\|\underline{x}\|_2=1}  (1+\underline{x}^{\top}\left(w \tilde{\Deltabf}_{1}\right) \underline{x})\leqslant 1+\left\|w \tilde{\Deltabf}_{1}\right\|_{F}. 
\end{equation}}
Since by Assumption 1 we have $\left\|\tilde{\Deltabf}_{1}\right\|_{F} \leqslant \sqrt{\alpha_{1}^{2} C^{2}+\beta_{1}^{2} C^{2}} \longrightarrow 0$,  for an arbitrarily fixed positive constant $\kappa$, $\left\|\tilde{\Deltabf}_{1}\right\|_{F} \leqslant \kappa$ and $\left\|\tilde{\Deltabf}_{2}\right\|_{F} \leqslant \kappa$ for large enough $n$. Then, by (\ref{T4eq1}) and (\ref{T4eq2}), for an arbitrary $\kappa > 0$,  $\Omegabf_{w}^{-1} \otimes \Omegabf_{w}^{-1} \succcurlyeq (1+\kappa)^{-2} \Omegabf_{0}^{-1} \otimes \Omegabf_{0}^{-1}$ in the Loewner ordering for large enough $n$. It follows that 
\begin{eqnarray}
 T_{4} &\geqslant& \frac{1}{p} \operatorname{vec}\left(\Deltabf_{1}\right)^{T}\left[\int_{0}^{1}\frac{1}{\left(1+\kappa\right)^{2}}\left(\Omegabf_{0}^{-1} \otimes \Omegabf_{0}^{-1}\right)(1-w) d w\right] \operatorname{vec}\left(\Deltabf_{1}\right) \nonumber\\
 &\geqslant& \frac{1}{p} \operatorname{vec}\left(\Deltabf_{1}\right)^{\top} \frac{1}{2\left(1+\kappa\right)^{2}}\left(\Omegabf_{0}^{-1} \otimes \Omegabf_{0}^{-1}\right) \operatorname{vec}\left(\Deltabf_{1}\right) \nonumber\\
  &=&\frac{1}{p} \operatorname{vec}\left(\Deltabf_{1}\right)^{\top} \frac{1}{2\left(1+\kappa\right)^{2}}\operatorname{vec}\left(\Omegabf_{0}^{-1}\Deltabf_{1}\Omegabf_{0}^{-1}\right) \nonumber\\
    &=&\frac{1}{2p\left(1+\kappa\right)^{2}}\operatorname{tr}\left(\Omegabf_{0}^{-1}\Deltabf_{1}\Omegabf_{0}^{-1}\Deltabf_{1}\right) \nonumber\\
&=&\frac{\left\|\tilde{\Deltabf}_{1}\right\|_{F}^{2}}{2 p\left(1+\kappa\right)^{2}} \nonumber\\
&=&\frac{1}{2 p\left(1+\kappa\right)^{2}}\left(C^{2} \alpha_{1}^{2}+C^{2} \beta_{1}^{2}\right) \label{Term4lower}
\end{eqnarray}

\noindent
for large enough $n$. Similarly, we get a lower bound for $T_{5}$, namely 
\begin{equation} \label{Term5lower}
T_{5} \geqslant \frac{1}{2 q\left(1+\kappa\right)^{2}}\left(C^{2} \alpha_{2}^{2}+C^{2} \beta_{2}^{2}\right). 
\end{equation}

\noindent
for large enough $n$. {\em The choice of the constant $\kappa$ will be made in the final stages of the proof.}

Now that we have provided positive lower bounds for $T_{4}$ and $T_{5}$, we proceed to bound $|T_{1}|\& |T_{2}|$, $|T_{7}|$ and $|T_{3}|$ individually so that their sum is bounded by $T_{4}+T_{5}+T_{6}$. We begin by analyzing $|T_{1}|$ and $|T_{2}|$. Recall that $\Xibf=\Yibf \Gammabf_{0}^{1 / 2}$, then the first term in $T_{1}$ is
$$
(n p q)^{-1} \sum_{i=1}^{n} \operatorname{tr}\left(\Yibf \Gammabf_{0} \Yibf^{\top} \Deltabf_{1}\right)=(n p q)^{-1} \sum_{i=1}^{n} \operatorname{tr}\left(\Xibf \Xibf^{\top} \Deltabf_{1}\right)=p^{-1} \operatorname{tr}\left(\Qbf_{1} \Deltabf_{1}\right).
$$
 Thus,
\begin{eqnarray*}
 T_{1}&=&p^{-1} \operatorname{tr}\left\{\left(\Qbf_{1}-\Sigmabf_{0}\right) \Deltabf_{1}\right\}=p^{-1}\left\{\sum_{(i, j) \in S_{1}}+\sum_{(i, j) \in S_{1}^{c}}\right\}\left(Q_{1}-\Sigma_{0}\right)_{i j}\left(\Delta_{1}\right)_{i j}\\
 &=:&T_{11}+T_{12}. 
\end{eqnarray*}
In other words, $T_{1}$ is decomposed into two components,  $T_{11}$ and $T_{12}$, by separating $\operatorname{tr}\left\{\left(\Qbf_{1}-\Sigmabf_{0}\right) \Deltabf_{1}\right\}$ into sums over entry indices in $S_1$ and $S_1^c$ respectively. Similarly, $T_{2}$ can be decomposed into two components,  $T_{21}$ and $T_{22}$,  by separating $\operatorname{tr}\left\{\left(\Qbf_{2}-\Psibf_{0}\right) \Deltabf_{2}\right\}$ into sums over the entry indices in $S_2$ and $S_2^c$ respectively, 
\begin{eqnarray*}
T_{2}&=&  q^{-1} \operatorname{tr}\left\{\left(\Qbf_{2}-\Psibf_{0}\right) \Deltabf_{2}\right\}=q^{-1}\left\{\sum_{(i, j) \in S_{2}}+\sum_{(i, j) \in S_{2}^{c}}\right\}\left(Q_{2}-\Psi_{0}\right)_{i j}\left(\Delta_{2}\right)_{i j}\\
&=:&T_{21}+T_{22}. 
\end{eqnarray*}
We first consider to bound the terms $|T_{11}|$ and $|T_{21}|$ by appropriate portions of $T_{4}$ and $T_{5}$ respectively. In order to achieve this, we first show that $\left\|\Deltabf_{1}\right\|_{F}$ can be upper bounded by a constant multiple of $\left\|\tilde{\Deltabf}_{1}\right\|_{F}$. In particular, by properties of trace operator and the positive definite matrix $\Omegabf_{0}$, we have

\begin{eqnarray}
\notag
\left\|\Deltabf_{1}\right\|_{F}&=&\sqrt{\operatorname{tr}\left(\Deltabf_{1} \Deltabf_{1}\right)}\\
\notag
&=&\left\{\operatorname{tr}\left(\Omegabf_{0}^{1 / 2} \tilde{\Deltabf}_{1} \Omegabf_{0} \tilde{\Deltabf}_{1} \Omegabf_{0}^{1 / 2}\right)\right\}^{1 / 2} \\
\notag
& \leqslant& \left\{\nu_{p}\left(\Omegabf_{0}\right)\operatorname{tr}\left(\Omegabf_{0}^{1 / 2} \tilde{\Deltabf}_{1} \tilde{\Deltabf}_{1} \Omegabf_{0}^{1 / 2}\right)\right\}^{1 / 2}\\
\notag
&\leqslant&\left\{\nu_{p}^{2}\left(\Omegabf_{0}\right)\operatorname{tr}\left( \tilde{\Deltabf}_{1}\tilde{\Deltabf}_{1} \right)\right\}^{1 / 2} \\
\label{transdelta}
& \leqslant& \frac{1}{\tau_{1}}\left\|\tilde{\Deltabf}_{1}\right\|_{F},
\end{eqnarray}
where inequality (\ref{transdelta}) follows from Assumption 2.
Similarly, $\left\|\Deltabf_{2}\right\|_{F}$ can be upper bounded by a constant multiple of $\left\|\tilde{\Deltabf}_{2}\right\|_{F}$, namely 
\begin{eqnarray}
\label{transdelta2}
   \left\|\Deltabf_{2}\right\|_{F} \leqslant \frac{1}{\tau_{1}}\left\|\tilde{\Deltabf}_{2}\right\|_{F}.
\end{eqnarray}
By equations (\ref{Q1}),(\ref{Term4lower}),(\ref{transdelta}) and the Cauchy–Schwarz  inequality, we get for large enough $n$
\begin{eqnarray}
\notag
 \left|T_{11}\right| &\leqslant& p^{-1}\left(s_{1}+p\right)^{1 / 2}\left\|\Deltabf_{1}\right\|_{F} \max _{i, j}\left|\left(Q_{1}-\Sigma_{0}\right)_{i j}\right| \\
 \notag
& \leqslant& \frac{K}{p} \sqrt{\frac{\left(s_{1}+p\right) \log p}{n q}}\left\|\Deltabf_{1}\right\|_{F}\\
\notag
&\leqslant& \frac{K}{p \tau_{1}} \sqrt{\alpha_{1}^{2}+\beta_{1}^{2}} \sqrt{C^{2} \alpha_{1}^{2}+C^{2} \beta_{1}^{2}} \\
\label{absT11}
& \leqslant& \frac{2 K\left(1+\kappa\right)^{2}}{\tau_{1} C} T_{4}.
\end{eqnarray}
Similarly,  by equations (\ref{Q2}),(\ref{Term5lower}),(\ref{transdelta2}) and the Cauchy–Schwarz  inequality, we get for large enough $n$, $|T_{21}|$ can be upper bounded by a constant multiple of $T_{5}$, namely
\begin{equation}
\label{absT21} \quad\left|T_{21}\right| \leqslant \frac{2 K\left(1+\kappa\right)^{2}}{\tau_{1} C} T_{5}.
\end{equation}
Next we bound $|T_{12}|$ and $|T_{22}|$ by an appropriate portion of $T_{6}$. Note that 
\begin{eqnarray*}
T_{6}=\sum_{(i, j) \in S_{1}^{c}} \lambda_{1}\left|\left(\Delta_{1}\right)_{i j}\right| +\sum_{(i, j) \in S_{2}^{c}} \lambda_{2}\left|\left(\Delta_{2}\right)_{i j}\right|=:T_{61}+T_{62}, 
\end{eqnarray*}
where $T_{61}:=\sum_{(i, j) \in S_{1}^{c}} \lambda_{1}\left|\left(\Delta_{1}\right)_{i j}\right|$ and $T_{62}:=\sum_{(i, j) \in S_{2}^{c}} \lambda_{2}\left|\left(\Delta_{2}\right)_{i j}\right|$. By inequality (\ref{Q1}), we have 

\begin{eqnarray*}
\left|T_{12}\right|&=&\left|\frac{1}{p} \sum_{(i, j) \in S_{1}^{c}}\left(Q_{1}-\Sigma_{0}\right)_{i j}\left(\Delta_{1}\right)_{i j}\right|\\
 &\leqslant& \frac{1}{p} \sum_{(i, j) \in S_{1}^{c}} K\sqrt{\frac{\log p}{n q}} \mid \left(\Delta_{1}\right)_{i j} \mid
\end{eqnarray*}
on $C_n$. Note by Assumption 4, we have $\left(\lambda_{1}-\frac{K}{p} \sqrt{\frac{\log p}{n q}}\right) \geqslant \frac{\lambda_{1}}{2}$ for large enough $n$. Thus,
\begin{equation}
\label{absT12}
    T_{61}-\left|T_{12}\right| \geqslant \sum_{(i, j) \in S_{1}^{c}}\left(\lambda_{1}-\frac{K}{p} \sqrt{\frac{\log p}{n q}}\right)\left|\left(\Delta_{1}\right)_{i j}\right| \geqslant \frac{\lambda_{1}}{2} \sum_{\left(i, j) \in S_{1}^{c}\right.}\left|\left(\Delta_{1}\right)_{i j}\right| = \frac{T_{61}}{2}
\end{equation}
for large enough $n$. We can similarly obtain

\begin{equation}
\label{absT22}
T_{62}-\left|T_{22}\right| \geqslant \frac{T_{62}}{2},
\end{equation}
for large enough $n$.
\noindent
Next, we proceed to bound $|T_{7}|$ with an appropriate portion of $T_{4}+T_{5}$. As with the term $T_{6}$, we divide $T_{7}$ into two parts as follows: 
\begin{eqnarray*}
    T_{7}&=&\sum_{(i, j) \in S_{1}, i \neq j} \left\{\lambda_{1}\left|\Omega_{ij}\right|-\lambda_{1}\left|\left(\Omega_{0}\right)_{i j}\right|\right\}+\sum_{(i, j) \in S_{2},i \neq j} \left\{\lambda_{2}\left|\Gamma_{ij}\right|-\lambda_{2}\left|\left(\Gamma_{0}\right)_{i j}\right|\right\}\\
    &=:&T_{71}+T_{72},
\end{eqnarray*}
where $T_{71}:=\sum_{(i, j) \in S_{1}, i \neq j} \left\{\lambda_{1}\left|\Omega_{ij}\right|-\lambda_{1}\left|\left(\Omega_{0}\right)_{i j}\right|\right\}$\\ and $T_{72}:=\sum_{(i, j) \in S_{2},i \neq j} \left\{\lambda_{2}\left|\Gamma_{ij}\right|-\lambda_{2}\left|\left(\Gamma_{0}\right)_{i j}\right|\right\}$. \\
Now note that for large enough $n$,
\begin{eqnarray}
\notag
\left|T_{71}\right| & \leqslant& \sum_{(i \neq j) \in S_{1}} \lambda_{1}\left|\Omega_{ij}-(\Omega_{0})_{i j}\right|=\lambda_{1} \sum_{(i \neq j) \in S_{1}} \mid(\Delta_{1})_{i j} \mid \\
\label{T71eq1}
& \leqslant& \lambda_{1} \sqrt{s_{1}} \sqrt{\sum_{(i, j) \in S_{1}}\left(\Delta_{1}\right)_{i j}^{2}}=\lambda_{1} \sqrt{s_{1}}\left\|\Deltabf_{1}\right\|_{F} \\
\label{T71eq2}
& \leqslant& \frac{\lambda_{1}}{\tau_{1}} \sqrt{s_{1}} \sqrt{C^{2} \alpha_{1}^{2}+C^{2} \beta_{1}^{2}} \\
\label{T71eq3}
& \leqslant& B \frac{\sqrt{s_{1}}}{\tau_{1}p}\sqrt{1+\frac{p}{s_{1}+1}} \sqrt{\frac{\log p}{n q}} \sqrt{C^{2} \alpha_{1}^{2}+C^{2} \beta_{1}^{2}} \\
\notag
& \leqslant&  \frac{B}{\tau_{1}p}\sqrt{s_{1}+p} \sqrt{\frac{\log p}{n q}} \sqrt{C^{2} \alpha_{1}^{2}+C^{2} \beta_{1}^{2}}\\
\label{absT71}
&\leqslant&\frac{2B(1+\kappa)^{2}}{\tau_{1}C}T_{4},
\end{eqnarray}
where inequality (\ref{T71eq1}) follows from the Cauchy-Schwarz inequality, inequality (\ref{T71eq2}) follows from (\ref{transdelta}), inequality (\ref{T71eq3}) follows from Assumption 3 (since that assumption implies the existence a constant $B$ such that $\lambda_{1} \leqslant \frac{B}{p}\sqrt{1+\frac{p}{s_{1}+1}} \sqrt{\frac{\log p}{n q}}$ for large enough $n$), and the last inequality follows from the lower bound for the term $T_{4}$ in (\ref{Term4lower}).

Similarly, by the Cauchy-Schwarz inequality, inequality (\ref{transdelta2}),  Assumption 3 and the lower bound for the term $T_5$ in (\ref{Term5lower}), we obtain 
\begin{equation}
\label{absT72}
 \left|T_{72}\right| \leqslant \frac{2B(1+\kappa)^{2}}{\tau_{1}C}T_{5}
\end{equation}
for large enough $n$.

Hence, we have bounds for $|T_1|$, $|T_2|$ and $|T_7|$ in terms of relevant portions of the terms $T_4$, $T_5$ and $T_6$. Lastly, we try to bound $T_{3}$ by the remaining portions of these three positive terms. By utilizing the properties of the trace and vectorization operators, along with the mixed-product property of the Kronecker product, we decompose $T_{3}$ into two components for separate bounding: 
\begin{eqnarray}
T_{3}&=&\frac{1}{n p q} \sum_{i=1}^{n} \operatorname{tr}\left(\Yibf \Deltabf_{2} \Yibf^{\top} \Deltabf_{1}\right) \nonumber\\
& =&\frac{1}{n p q} \sum_{i=1}^{n} \operatorname{vec}\left(\Yibf^{\top}\right)^{\top} \operatorname{vec}\left(\Deltabf_{2} \Yibf^{\top} \Deltabf_{1}\right)  \nonumber\\
& =&\frac{1}{n p q} \sum_{i=1}^{n} \operatorname{vec}\left(\Yibf^{\top}\right)^{\top}\left(\Deltabf_{1} \otimes \Deltabf_{2}\right) \operatorname{vec}\left(\Yibf^{\top}\right) \nonumber\\
& =&\frac{1}{n p q} \sum_{i=1}^{n} \operatorname{tr}\left(\operatorname{vec}\left(\Yibf^{\top}\right)^{\top}\left(\Deltabf_{1} \otimes \Deltabf_{2}\right) \operatorname{vec}\left(\Yibf^{\top}\right)\right) \nonumber\\
& =&\frac{1}{n p q} \sum_{i=1}^{n} \operatorname{tr}\left(\left(\Deltabf_{1} \otimes \Deltabf_{2}\right) \operatorname{vec}\left(\Yibf^{\top}\right) \operatorname{vec}\left(\Yibf^{\top}\right)^{\top}\right) \nonumber\\
& =&\frac{1}{p q} \operatorname{tr}\left(\left(\Deltabf_{1} \otimes \Deltabf_{2}\right)\left(\frac{1}{n} \sum_{i=1}^{n} \operatorname{vec}\left(\Yibf^{\top}\right) \operatorname{vec}\left(\Yibf^{\top}\right)^{\top}\right)\right) \nonumber\\
& =&\frac{1}{p q} \operatorname{tr}\left(\left(\Deltabf_{1} \otimes \Deltabf_{2}\right)\left(\Sbf-\Sigmabf_{0} \otimes \Psibf_{0}\right)\right)+\frac{1}{p q} \operatorname{tr}\left(\left(\Deltabf_{1} \otimes \Deltabf_{2}\right)\left(\Sigmabf_{0} \otimes \Psibf_{0}\right)\right) \nonumber\\
& =:&T_{31}+T_{32}, \label{t3:decomposition}
\end{eqnarray}
where {\small$T_{31}:=\frac{1}{p q} \operatorname{tr}\left(\left(\Deltabf_{1} \otimes \Deltabf_{2}\right)\left(\Sbf-\Sigmabf_{0} \otimes \Psibf_{0}\right)\right)$} and {\small$T_{32}:=\frac{1}{p q} \operatorname{tr}\left(\left(\Deltabf_{1} \otimes \Deltabf_{2}\right)\left(\Sigmabf_{0} \otimes \Psibf_{0}\right)\right)$}.
Note that 
\begin{eqnarray}
\notag
|\operatorname{tr}\left(\Deltabf_{1} \Sigmabf_{0}\right)|&=&|\operatorname{tr}\left(\Sigmabf_{0}^{1/2}\Deltabf_{1} \Sigmabf_{0}^{1/2}\right)|=|\operatorname{tr}(\tilde{\Deltabf}_1)| \\
\label{T31tr}
&\leqslant& \sqrt{p}\sqrt{\sum_{r=1}^p (\tilde{\Delta}_1)_{rr}^2}
=\sqrt{p} \beta_{1} C=C p \sqrt{\frac{\log p}{n q}}. 
\end{eqnarray}
Similarly, we have $\left|\operatorname{tr}\left(\Deltabf_{2} \Psibf_{0}\right)\right| \leqslant \sqrt{q} \beta_{2} C=C q \sqrt{\frac{\log q}{n p}}$. Combined with (\ref{Term4lower}), (\ref{Term5lower}) and the mixed-product property of the Kronecker product, for large enough $n$, we can bound $\left|T_{32}\right|$ as follows,
\begin{eqnarray}
\notag
\left|T_{32}\right| &=&\left|\frac{1}{pq} \operatorname{tr}\left(\Deltabf_{1} \Sigmabf_{0}\right) \operatorname{tr}\left(\Deltabf_{2} \Psibf_{0}\right)\right|\\
\notag
&\leqslant& C^{2} \sqrt{\frac{\log p}{n q}} \sqrt{\frac{\log q}{n p}} \leqslant C^{2}\left(\frac{1}{2} \frac{\log p}{n q}+\frac{1}{2} \frac{\log q}{n p}\right)\\
&=&C^{2}\left(\frac{\beta_{1}^{2}}{2 p}+\frac{\beta_{2}^{2}}{2 q}\right) 
\leqslant \frac{C^{2}}{2}\left(\frac{\alpha_{1}^{2}+\beta_{1}^{2}}{2 p}+\frac{\alpha_{2}^{2}+\beta_{2}^{2}}{2 q}\right)\\
\label{absT32}
& \leqslant &\frac{\left(1+\kappa\right)^{2}}{2}\left(T_{4}+T_{5}\right),
\end{eqnarray}
which indicates $|T_{32}|$ can also be upper bounded by some appropriate portion of $T_{4}+T_{5}$.

Before we analyze the last remaining term, $|T_{31}|$, we collect our bounds so far from (\ref{absT11}), (\ref{absT21}), (\ref{absT12}), (\ref{absT22}), (\ref{absT71}), (\ref{absT72}), (\ref{absT32}) to get the following inequality on the event $C_n$ for large enough $n$: 
\begin{eqnarray}
\label{woT31}
|T_{1}|+|T_{2}|+|T_{7}|+|T_{32}|
\leqslant \left(\frac{2 K}{\tau_{1} C}+\frac{2B}{\tau_{1} C}+\frac{1}{2}\right) (1+\kappa)^{2}(T_{4}+ T_{5})+ \frac{1}{2}T_{6}
\end{eqnarray}
Note that the constants $\kappa$ and $C$ are arbitrary. Hence, the coefficient $(\frac{2 K}{\tau_{1} C}+\frac{2B}{\tau_{1} C}+\frac{1}{2})(1+\kappa)^{2}$ for $T_4$ and $T_5$ can be adjusted to approach $1/2$ by selecting large enough $C$ and small enough $\kappa$. 

To analyze $|T_{31}|$, we first rewrite $T_{31}$ in an alternative form. In particular, note that 
{\footnotesize
\begin{eqnarray*}
T_{31}&=&\frac{1}{p q} \operatorname{tr}\left(\left(\Deltabf_{1} \otimes \Deltabf_{2}\right)\left(\Sbf-\Sigmabf_{0} \otimes \Psibf_{0}\right)\right)\\
& =&\frac{1}{p q} \sum_{j_{1},  j_{4}=1}^{p} \sum_{j_{2}, j_{3}=1}^{q}\left(\Delta_{2}\right)_{j_{2} j_{3}}\left(\Delta_{1}\right)_{j_{1} j_{4}}\left(\frac{1}{n} \sum_{i=1}^{n}\left(Y_{i}\right)_{j_{1} j_{2}}\left(Y_{i}\right)_{j_{4} j_{3}}-\left(\Sigma_{0}\right)_{j_{1} j_{4}}\left(\Psi_{0}\right)_{j_{2} j_{3}} \right). 
\end{eqnarray*}}
Here, $E_{P_0}\left[\left(Y_{i}\right)_{j_{1} j_{2}}\left(Y_{i}\right)_{j_{4} j_{3}}\right]=\left(\Sigma_{0}\right)_{j_{1} j_{4}}\left(\Psi_{0}\right)_{j_{2} j_{3}}$. 
By (\ref{Sk}), on $C_n$, we have
{\scriptsize
\begin{eqnarray} \label{T31U}
\left|T_{31}\right| 
&\leqslant& \frac{K}{p q} \sqrt{\frac{\log p q}{n}}\left(\sum_{j_{1}, j_{4}=1}^{p}\left|\left(\Delta_{1}\right)_{j_{1} j_{4}}\right|\right)\left(\sum_{j_{2}, j_{3}=1}^{q}\left|\left(\Delta_{2}\right)_{j_{2} j_{3}}\right|\right) \nonumber\\
&=&A_{n}\left(\sum_{j_{1}=j_{4}}\left|\left(\Delta_{1}\right)_{j_{1}j_{4}}\right|\right)\left(\sum_{j_{2}=j_{3}}\left|\left(\Delta_{2}\right)_{j_{2}j_{3}}\right|\right) \nonumber\\
&+&A_{n}\left(\sum_{j_{1}=j_{4}}\left|\left(\Delta_{1}\right)_{j_{1} j_{4}}\right|\right)\left(\sum_{j_{2} \neq j_{3}}\left|\left(\Delta_{2}\right)_{j_{2} j_{3}}\right|\right)+A_{n}\left(\sum_{j_{1}\neq j_{4}}\left|\left(\Delta_{1}\right)_{j_{1} j_{4}}\right|\right)\left(\sum_{j_{2} = j_{3}}\left|\left(\Delta_{2}\right)_{j_{2} j_{3}}\right|\right)
\nonumber\\
&+&A_{n}\left(\sum_{j_{1}\neq j_{4}}\left|\left(\Delta_{1}\right)_{j_{1} j_{4}}\right|\right)\left(\sum_{j_{2} \neq j_{3}}\left|\left(\Delta_{2}\right)_{j_{2} j_{3}}\right|\right) \nonumber\\
&=:& U_1+U_2+U_3,
\end{eqnarray}
}
where we set $A_{n}:=\frac{K}{p q} \sqrt{\frac{\log p q}{n}}$ for simplicity of notation, and 
{\scriptsize
\begin{eqnarray*}
U_1&:=&A_{n}\left(\sum_{j_{1}=j_{4}}\left|\left(\Delta_{1}\right)_{j_{1}j_{4}}\right|\right)\left(\sum_{j_{2}=j_{3}}\left|\left(\Delta_{2}\right)_{j_{2}j_{3}}\right|\right)\\
U_2&:=&A_{n}\left(\sum_{j_{1}=j_{4}}\left|\left(\Delta_{1}\right)_{j_{1} j_{4}}\right|\right)\left(\sum_{j_{2} \neq j_{3}}\left|\left(\Delta_{2}\right)_{j_{2} j_{3}}\right|\right)+A_{n}\left(\sum_{j_{1}\neq j_{4}}\left|\left(\Delta_{1}\right)_{j_{1} j_{4}}\right|\right)\left(\sum_{j_{2} = j_{3}}\left|\left(\Delta_{2}\right)_{j_{2} j_{3}}\right|\right)\\
U_3&:=&A_{n}\left(\sum_{j_{1}\neq j_{4}}\left|\left(\Delta_{1}\right)_{j_{1} j_{4}}\right|\right)\left(\sum_{j_{2} \neq j_{3}}\left|\left(\Delta_{2}\right)_{j_{2} j_{3}}\right|\right).    
\end{eqnarray*}}
 We deal with these three parts separately. First, note that for large enough $n$
\begin{eqnarray}
\notag
U_1&=&A_{n}\left(\sum_{j_{1}=j_{4}}\left|\left(\Delta_{1}\right)_{j_{1}j_{4}}\right|\right)\left(\sum_{j_{2}=j_{3}}\left|\left(\Delta_{2}\right)_{j_{2}j_{3}}\right|\right)\\
\label{U1eq1}
& \leqslant& \frac{K}{p q} \sqrt{\frac{\log p q}{n}} \cdot \sqrt{p}\left\|\Delta_{1}\right\|_{F} \sqrt{q}\left\|\Delta_{2}\right\|_{F} \\
\label{U1eq2}
&\leqslant&\frac{K}{\tau_{1}^{2} p q} \sqrt{\frac{\log p q}{n}} \sqrt{p} \sqrt{C^{2} \alpha_{1}^{2}+C^{2} \beta_{1}^{2}} \sqrt{q} \sqrt{C^{2} \alpha_{2}^{2}+C^{2} \beta_{2}^{2}} \\
\label{U1eq3}
& \leqslant& \frac{K}{\tau_{1}^{2} p q} \sqrt{\frac{\log p q}{n}}\left[\frac{q}{2}\left(C^{2} \alpha_{1}^{2}+C^{2} \beta_{1}^{2}\right)+\frac{p}{2}\left(C^{2} \alpha_{2}^{2}+C^{2} \beta_{2}^{2}\right)\right] \\
\label{U1eq4}
& \leqslant&\frac{K(1+\kappa)^{2}}{\tau_{1}^{2}} \sqrt{\frac{\log p q}{n}}\left(T_{4}+T_{5}\right)
\end{eqnarray}
where inequality (\ref{U1eq1}) follows from the Cauchy-Schwarz inequality, inequality (\ref{U1eq2}) follows from (\ref{transdelta}) and (\ref{transdelta2}), inequality (\ref{U1eq3}) follows from the AM-GM inequality, and inequality (\ref{U1eq4}) follows from the lower bounds for the term $T_{4}$ and $T_{5}$ in (\ref{Term4lower}) and (\ref{Term5lower}).

Next, we proceed to bound $U_2$. In particular, note that the first term in $U_2$ satisfies 
\begin{eqnarray}
\notag
&&A_{n}\left(\sum_{j_{1}=j_{4}}\left|\left(\Delta_{1}\right)_{j_{1} j_{4}}\right|\right)\left(\sum_{j_{2} \neq j_{3}}\left|\left(\Delta_{2}\right)_{j_{2} j_{3}}\right|\right)\\
\notag
&=&A_{n} \left(\sum_{j_{1}=j_{4}}\left|\left(\Delta_{1}\right)_ {j_{1} j_{4}}\right|\right)\left(\sum_{\left(j_{2} \neq j_{3}\right) \in S_{2}}\left|\left(\Delta_{2}\right)_{j_{2}  j_{3}}\right|+\sum_{\left(j_{2} \neq j_{3}\right) \in S_{2}^c}\left|\left(\Delta_{2}\right)_{j_{2}  j_{3}}\right|\right) \\
\label{U21eq1}
& \leqslant& A_{n} \sqrt{p}\left\|\Deltabf_{1}\right\|_{F} \sqrt{s_{2}}\left\|\Deltabf_{2}\right\|_{F}+A_{n} \sqrt{p}\left\|\Deltabf_{1}\right\|_{F} \sum_{\left(j_{2} \neq j_{3}\right) \in S^{c}}\left|\Gamma_{j_{2} j_{3}}\right|,
\end{eqnarray}
where (\ref{U21eq1}) follows from the Cauchy-Schwarz inequality. The first term in (\ref{U21eq1}) can now be bounded as 
\begin{eqnarray}
\notag
&&A_{n} \sqrt{p}\left\|\Deltabf_{1}\right\|_{F} \sqrt{s_{2}}\left\|\Deltabf_{2}\right\|_{F}\\
\label{U21eq2}
& \leqslant& \frac{K}{\tau_{1}^{2} p q} \sqrt{\frac{\log p q}{n}} \sqrt{p} \sqrt{C^{2} \alpha_{1}^{2}+C^{2} \beta_{1}^{2}} \sqrt{s_{2}} \sqrt{C^{2} \alpha_{2}^{2}+C^{2} \beta_{2}^{2}} \\
\label{U21eq3}
& \leqslant& \frac{K}{\tau_{1}^{2} } \sqrt{\frac{\log p q}{n}} \sqrt{\frac{s_{2}}{q}}\left[\frac{1}{2p}\left(C^{2} \alpha_{1}^{2}+C^{2} \beta_{1}^{2}\right)+\frac{1}{2q}\left(C^{2} \alpha_{2}^{2}+C^{2} \beta_{2}^{2}\right)\right] \\
\label{U21eq4}
& \leqslant& \frac{K(1+\kappa)^{2}}{\tau_{1}^{2}} \sqrt{\frac{\log p q}{n}} \sqrt{\frac{s_{2}}{q}} \left(T_{4}+T_{5}\right) 
\end{eqnarray}
where inequality (\ref{U21eq2}) follows from (\ref{transdelta}) and (\ref{transdelta2}), inequality (\ref{U21eq3}) follows from the AM-GM inequality, and (\ref{U21eq4}) follows from the lower bounds for the term $T_{4}$ and $T_{5}$ in (\ref{Term4lower}) and (\ref{Term5lower}). Subsequently, for the second term in (\ref{U21eq1}), the inequality in (\ref{transdelta}) implies that 
{\small
\begin{equation} \label{U21eq6}
A_{n} \sqrt{p}\left\|\Deltabf_{1}\right\|_{F}\sum_{\left(j_{2} \neq j_{3}\right) \in S_{2}^{c}}\left|\Gamma_{j_{2} j_{3}}\right| 
\leqslant \frac{K C}{\tau_{1} p q} \sqrt{\frac{\log p q}{n}} \sqrt{p} \sqrt{\frac{\left(p+\zeta_{1}\right) \log p}{n q}}\sum_{\left(j_{2} \neq j_{3}\right) \in S_{2}^{c}}\left|\Gamma_{j_{2} j_{3}}\right|. 
\end{equation}}
Following a similar line of reasoning, the second term of $U_2$ can be bounded as 

{\footnotesize
\begin{eqnarray}
\notag
&&A_{n}\left(\sum_{j_{1}\neq j_{4}}\left|\left(\Delta_{1}\right)_{j_{1} j_{4}}\right|\right)\left(\sum_{j_{2} =j_{3}}\left|\left(\Delta_{2}\right)_{j_{2} j_{3}}\right|\right)\\
\notag
&\leqslant&\frac{K(1+\kappa)^{2}}{\tau_{1}^{2}} \sqrt{\frac{\log p q}{n}} \sqrt{\frac{s_{1}}{p}} \left(T_{4}+T_{5}\right) +\frac{K C}{\tau_{1} p q} \sqrt{\frac{\log p q}{n}}\sqrt{q}  \sqrt{\frac{\left(q+\zeta_{2}\right) \log q}{np }}\sum_{\left(j_{1} \neq j_{4}\right) \in S_{1}^{c}}\left|\Omega_{j_{1} j_{4}}\right|. \\
\label{U2eq1}
\end{eqnarray}}
Finally, note that 
{\footnotesize
\begin{eqnarray}
\notag
U_3&=&A_{n}\left(\sum_{j_{1}\neq j_{4}}\left|\left(\Delta_{1}\right)_{j_{1} j_{4}}\right|\right)\left(\sum_{j_{2} \neq j_{3}}\left|\left(\Delta_{2}\right)_{j_{2} j_{3}}\right|\right)\\
\notag
&=&A_{n} \left(\sum_{j_{1} \neq j_{4} }\left|\left(\Delta_{1}\right)_{j_{1}  j_{4}}\right|\right)\left(\sum_{\left(j_{2} \neq j_{3}\right) \in S_{2}}\left|\left(\Delta_{2}\right)_{j_{2}  j_{3}}\right|+\sum_{\left(j_{2} \neq j_{3}\right) \in S_{2}^c}\left|\left(\Delta_{2}\right)_{j_{2}  j_{3}}\right|\right) \\
\label{U3eq1}
&\leqslant & A_{n}\left(p\left\|\Deltabf_{1}\right\|_{F} \sqrt{s_{2}}\left\|\Deltabf_{2}\right\|_{F}+p\left\|\Deltabf_{1}\right\|_{F} \sum_{\left(j_{2} \neq j_{3}\right) \in S_{2}^{c}}\left|\Gamma_{j_{2} j_{3}}\right|\right) \\
\notag
& \leqslant& \frac{K(1+\kappa)^{2}}{\tau_{1}^{2}} \sqrt{\frac{\log p q}{n}} \sqrt{\frac{s_{2}p}{q}} \left(T_{4}+T_{5}\right)+\frac{K C}{\tau_{1}  q} \sqrt{\frac{\log p q}{n}} \sqrt{\frac{\left(p+\zeta_{1}\right) \log p}{n q}}\sum_{\left(j_{2} \neq j_{3}\right) \in S_{2}^{c}}\left|\Gamma_{j_{2} j_{3}}\right|\\
\label{U3eq2}
\end{eqnarray}}
where (\ref{U3eq1}) follows from the Cauchy-Schwarz inequality, and (\ref{U3eq2}) follows from (\ref{U21eq4}) and (\ref{U21eq6}) by multiplying $\sqrt{p}$ on both sides of those inequalities. Similarly, by decomposing the first factor of $U_{3}$ instead of the second factor and multiplying $\sqrt{q}$ on both sides of inequality (\ref{U2eq1}), we obtain
{\footnotesize
\begin{equation}
\label{U3eq3}
U_{3}\leqslant\frac{K(1+\kappa)^{2}}{\tau_{1}^{2}} \sqrt{\frac{\log p q}{n}} \sqrt{\frac{s_{1}q}{p}} \left(T_{4}+T_{5}\right) +\frac{K C}{\tau_{1} p } \sqrt{\frac{\log p q}{n}}  \sqrt{\frac{\left(q+\zeta_{2}\right) \log q}{np }}\sum_{\left(j_{1} \neq j_{4}\right) \in S_{1}^{c}}\left|\Omega_{j_{1} j_{4}}\right|. 
\end{equation}}

By combining all the bounds above for $U_1$, $U_2$ and $U_3$ along with Assumption 5 (using bound (\ref{U3eq3}) for $U_3$ if $ \frac{r_n\log p q}{n} \longrightarrow 0$ and using bound (\ref{U3eq2}) for $U_3$ if $ \frac{r_n'\log p q}{n} \longrightarrow 0$), it follows that there exists a sequence of constants 
$\{c_{2,n}\}_{n \geqslant 1}$ with $c_{2,n} \rightarrow 0$ as $n \rightarrow \infty$,
such that 
\begin{equation} \label{T31bound}
|T_{31}|\leqslant U_1+U_2+U_3 \leqslant c_{2,n} (T_4 + T_5 + T_6) \end{equation}

\noindent
on the event $C_n$ for large enough $n$. Combining this with the bound in (\ref{woT31}), we obtain
{\footnotesize
\begin{eqnarray*}
&&\inf _{\tilde{\Deltabf}_{1} \in \mathcal{A}_p, \tilde{\Deltabf}_{2} \in \mathcal{B}_q}\left\{g(\Omegabf, \Gammabf)-g\left(\Omegabf_{0}, \Gammabf_{0}\right)\right\}\\
&=&\inf _{\tilde{\Deltabf}_{1} \in \mathcal{A}_p, \tilde{\Deltabf}_{2} \in \mathcal{B}_q}\left\{T_{1}+T_{2}+T_{3}+T_{4}+T_{5}+T_{6}+T_{7}\right\}\\
&\geqslant& \inf _{\tilde{\Deltabf}_{1} \in \mathcal{A}_p, \tilde{\Deltabf}_{2} \in \mathcal{B}_q}\left\{\left(1-c_{2,n}-\left(\frac{2 K}{\tau_{1} C}+\frac{2B}{\tau_{1} C}+\frac{1}{2}\right) (1+\kappa)^{2}\right)(T_{4}+ T_{5})+ (\frac{1}{2}-c_{2,n})T_{6} \right\}
\end{eqnarray*}}
on event $C_n$ for large enough $n$.
\noindent
Now choose the constant $C = \frac{16 K}{\tau_{1}}+\frac{16B}{\tau_{1}}$ and the constant $\kappa = 0.01$. Since $c_{2,n} \longrightarrow 0$, it follows that $c_{2,n} \leqslant \frac{1}{4}$ eventually. Hence 
\begin{eqnarray}
\notag
&&\inf_{\tilde{\Deltabf}_{1} \in \mathcal{A}_p, \tilde{\Deltabf}_{2} \in \mathcal{B}_q}\left\{g(\Omegabf, \Gammabf)-g\left(\Omegabf_{0}, \Gammabf_{0}\right)\right\}\\
\notag
&\geqslant& \inf _{\tilde{\Deltabf}_{1} \in \mathcal{A}_p, \tilde{\Deltabf}_{2} \in \mathcal{B}_q} \left\{ 0.112 (T_4 + T_5) \right\}\\
\notag
&\geqslant& 0.112 \left( \frac{1}{2 p\left(1+\kappa\right)^{2}}\left(C^{2} \alpha_{1}^{2}+C^{2} \beta_{1}^{2}\right)+\frac{1}{2 q\left(1+\kappa\right)^{2}}\left(C^{2} \alpha_{2}^{2}+C^{2} \beta_{2}^{2}\right)\right)\\
\label{gdiffpositive}
&>& 0
\end{eqnarray}
on $C_n$ for large enough $n$.
\noindent
Since $P_0 (C_n) \rightarrow 1 \mbox{ as } n \rightarrow \infty$, we have thereby proved (\ref{identity}), and have established the required 
result. 
\end{proof}

\section{Errors in previous consistency proofs}\label{Error}

\noindent
In this section, we carefully examine the errors in the previous 
consistency arguments for the Sparse SMGM estimator, and also highlight the relevant features/innovations in our argument from Section \ref{high:dim:consistency}.

\subsection{The proof of \texorpdfstring{\cite[Theorem 1]{leng2012sparse}}{Theorem 1 in Leng (2012)}}

\noindent
We start by explaining the mistakes and lapses in the consistency proof of the SMGM estimator in \cite[Theorem 1]{leng2012sparse}. As in Section \ref{high:dim:consistency}, the authors in \cite{leng2012sparse} also pursue the high level strategy of showing that with high probability the infimum of the objective function $g$ over an appropriate set is strictly larger than $g(\Omegabf_0, \Gammabf_0)$. There is however, a key and consequential difference. While the argument in Section \ref{high:dim:consistency} considers infimum over the set $\tilde{\Deltabf}_{1} \in \mathcal{A}_p, \tilde{\Deltabf}_{2} \in \mathcal{B}_q$, the proof in \cite{leng2012sparse} considers infimum over the set $\Deltabf_{1} \in \mathcal{A}_p', \Deltabf_{2} \in \mathcal{B}_q'$, where 
{\small
\begin{eqnarray*}
\mathcal{A}_p' &:=& \left\{\mathbf{M}: \; \mathbf{M}=\alpha_{1}' \Rbf_{p}+\beta_{1} \Dbf_{p}, \; 
\Dbf_p \in \mathcal{D}_p, \; \Rbf_p \in \mathcal{R}_p \mbox{ and } \|\Dbf_p\|_F = \|\Rbf_p\|_F = C \right\}, \\
\mathcal{B}_q' &:=& \left\{\mathbf{M}: \; \mathbf{M}=\alpha_{2}' \Rbf_{q}+\beta_{2} \Dbf_{q}, \; 
\Dbf_q \in \mathcal{D}_q, \; \Rbf_q \in \mathcal{R}_q \mbox{ and } \|\Dbf_q\|_F = \|\Rbf_q\|_F = C \right\}. 
\end{eqnarray*}
and
\begin{eqnarray*}
\alpha_{1}'=\left\{s_{1} \log p /(n q)\right\}^{1 / 2},\alpha_{2}'=\left\{s_{2} \log q /(n p)\right\}^{1 / 2}.
\end{eqnarray*}}
As in the proof of Theorem \ref{Fnormcr}, the difference $g\left(\Omegabf_{0}+\Deltabf_{1}, \Gammabf_{0}+\Deltabf_{2}\right)-g\left(\Omegabf_{0}, \Gammabf_{0}\right)$ is expressed as the sum $T_1 + T_2 + \cdots + T_7$, and the strategy is to use the non-negative terms $T_4, T_5$ and $T_6$ to control/bound the other terms. The errors in the proof of \cite[Theorem 1]{leng2012sparse} arise chiefly in the analysis/bounding of the term $T_3$. Similar to (\ref{t3:decomposition}), the term $T_3$ is broken up as a sum of two parts $T_{31}$ and $T_{32}$ in the proof of \cite[Theorem 1]{leng2012sparse}, and these terms are bounded using separate arguments. We examine both these arguments below. 

\subsubsection{Bound for \texorpdfstring{$T_{32}$}{T32}} \label{T32}

\noindent
The authors in \cite{leng2012sparse} employ 
the following strategy to bound the term $|T_{32}|$. First, the following lower bounds are established for $T_4$ and $T_5$. 
\begin{eqnarray}
\label{an}
&&T_{4} \geqslant(2 p)^{-1}\left\{\tau_{1}^{-1}+a_n\right\}^{-2}\left(C^{2} {\alpha_1'}^{2}+C^{2} \beta_{1}^{2}\right)\\
\label{bn}
&&T_{5} \geqslant(2 q)^{-1}\left\{\tau_{1}^{-1}+b_n\right\}^{-2}\left(C^{2} {\alpha_2'}^{2}+C^{2} \beta_{2}^{2}\right),
\end{eqnarray}
where $a_n$ and $b_n$ are constants which converge to $0$ as $n \rightarrow \infty$. Then, an upper bound for $T_{32}$ in terms of $T_4$ and $T_5$ is derived by noting that for $\Deltabf_{1} \in \mathcal{A}_p'$ and $\Deltabf_{2} \in \mathcal{B}_q'$, 
\begin{eqnarray}
\notag
   |T_{32}| &\leqslant &1/\sqrt{pq}\tau_1^{-2} (C^{2} {\alpha_1'}^{2}+C^{2} \beta_{1}^{2})(C^{2} {\alpha_2'}^{2}+C^{2} \beta_{2}^{2}) \\
   \notag
   &\leqslant & (2p)^{-1}\tau_1^{-2} (C^{2} {\alpha_1'}^{2}+C^{2} \beta_{1}^{2})+(2q)^{-1}\tau_1^{-2}(C^{2} {\alpha_2'}^{2}+C^{2} \beta_{2}^{2}) \\
    \label{tau_1^2}
    &\leqslant& T_4 + T_5
\end{eqnarray}
with probability converging to 1 as $n \rightarrow \infty$.
\noindent
This line of attack has the following problems. 
\begin{itemize}
    \item The inequality (\ref{tau_1^2}) is erroneous for two reasons. First, the constants $a_n$ and $b_n$ are in the lower bounds (\ref{an}) and (\ref{bn}) for $T_4, T_5$ are ignored. While these constants vanish in the limit, they cannot be ignored in the above bounds which need to hold {\it for every large enough $n$}. Second, the $\tau_1$ related terms in the lower bounds (\ref{an}) and (\ref{bn}) are close to $\tau_1^2$, whereas the relevant $\tau_1$ related terms in the  upper bound for $|T_{32}|$ are $\tau_1^{-2}$. By definition, $\tau_1 \leqslant \tau_1^{-1}$, where a strict inequality $\tau_1 < \tau_1^{-1}$ holds unless $\Sigmabf_0 = {\bf I}_p$ and $\Psibf_0 = {\bf I}_q$. Hence, even if $a_n$ and $b_n$ are ignored, the last inequality in (\ref{tau_1^2}) holds in the opposite direction. 
    \item Even if (\ref{tau_1^2}) was accurate, the entirety of the terms $T_4$ and $T_5$ have now been used to control $T_{32}$, leaving no portions of these terms to control other terms like $T_{31}$, $T_1$, $T_2$ and $T_7$. 
    \item A strict inequality is needed in (\ref{tau_1^2}) instead of a less than or equal to condition. 
\end{itemize}

\noindent
We now discuss the key elements of our argument (Section \ref{high:dim:consistency}) that enable us to avoid the errors delineated above. A salient innovation in our proof is the use of the neighborhood $\tilde{\Deltabf}_{1} \in \mathcal{A}_p, \tilde{\Deltabf}_{2} \in \mathcal{B}_q$ instead of $\Deltabf_{1} \in \mathcal{A}_p', \Deltabf_{2} \in \mathcal{B}_q'$ to search for the local minimizer (recall that $\tilde{\Deltabf}_{1}=\Omegabf_{0}^{-1 / 2} \Deltabf_{1} \Omegabf_{0}^{-1 / 2}, \tilde{\Deltabf}_{2}=\Gammabf_{0}^{-1 / 2} \Deltabf_{2} \Gammabf_{0}^{-1 / 2}$). This approach provides a much tighter inequality for upper bounding $|T_{32}|$ using a significantly smaller portion of $T_{4}$ \& $T_{5}$, and thereby helps prevent the relevant errors.

In particular, note that $T_{32}=\frac{1}{pq} \operatorname{tr}\left(\Deltabf_{1} \Sigmabf_{0}\right) \operatorname{tr}\left(\Deltabf_{2} \Psibf_{0}\right)$. If as in \cite{leng2012sparse}, one uses the constraint $\Deltabf_{1} \in \mathcal{A}_p'$, the 
tightest upper bound for $\operatorname{tr}\left(\Deltabf_{1} \Sigmabf_{0}\right)$ is obtained by  $$|\operatorname{tr}\left(\Deltabf_{1}\Sigmabf_{0}\right)|\leqslant \tau_1^{-1}p^{1/2}\|\Deltabf_1\|_F\leqslant 
 \tau_1^{-1}p^{1/2}(C^2{\alpha_1'}^{2}+C^2\beta_1^2)^{1/2}.$$
 The introduction of the constant $\tau_1^{-1}$ is unavoidable in the upper bound as $\Deltabf_1$ and $\Sigmabf_0$ need to be controlled separately in this setup. On the other hand, by working with the constraint $\tilde{\Deltabf}_1 \in \mathcal{A}_p$, we are able to show $$|\operatorname{tr}\left(\Deltabf_{1}\Sigmabf_{0}\right)|=|\operatorname{tr}(\tilde{\Deltabf}_1)|=p^{1/2}C \beta_1,$$ 
 which is an exact equality and avoids introducing the factor $\tau_1^{-1}$ in the bound. A similar analysis holds for the bounds for the term $\operatorname{tr}\left(\Deltabf_{2} \Psibf_{0}\right)$. 
 At the same time, authors in \cite{leng2012sparse}, under the constraint 
 $\Deltabf_1 \in \mathcal{A}_p$, obtain a lower bound for $T_4$ as 
 \begin{eqnarray}
 \notag
 T_{4} & \geqslant& p^{-1}\left\|\operatorname{vec}\left(\boldsymbol{\Delta}_{1}\right)\right\|_{2}^{2} \int_{0}^{1}(1-w) \min _{0<w<1} \nu_{1}\left(\boldsymbol{\Omega}_{w}^{-1} \otimes \boldsymbol{\Omega}_{w}^{-1}\right) d w \\ 
 \notag
 & \geqslant& (2 p)^{-1}\left\|\operatorname{vec}\left(\boldsymbol{\Delta}_{1}\right)\right\|_{2}^{2} \min _{0<w<1} \nu_{p}^{-2}\left(\boldsymbol{\Omega}_{w} \right)\\
 \notag
 & \geqslant& (2 p)^{-1}\left\|\operatorname{vec}\left(\boldsymbol{\Delta}_{1}\right)\right\|_{2}^{2} (\|\Omegabf_0\|+\|\Deltabf_1\|)^{-2}\\
 \label{T42terms}
&\geqslant&(2 p)^{-1}\left(C^{2} {\alpha_1'}^{2}+C^{2} \beta_{1}^{2}\right)\left\{\tau_{1}^{-1}+a_n\right\}^{-2}.  
 \end{eqnarray}
 Here, for a symmetric matrix $\Abf$, we use $\|\Abf\|$ to denote the spectral norm of $\Abf$, and $a_n$ is $o(1)$. Again, the introduction of the term $\tau_1^{-1}$ is unavoidable since the terms 
$\left\|\operatorname{vec}\left(\boldsymbol{\Delta}_{1}\right)\right\|_{2}^{2}$ and $ (\|\Omegabf_0\|+\|\Deltabf_1\|)^{-2}$ need to be controlled separately in this setup. On the other hand, in our setup, with $\tilde{\Deltabf}_{1}\in \mathcal{A}_p$, first, as in (\ref{T4eq1}), one can express 
 \begin{equation}
\Omegabf_{w}^{-1} \otimes \Omegabf_{w}^{-1}
=(\Omegabf_{0}^{-1 / 2}\left(\Ibf+w \tilde{\Deltabf}_{1}\right)^{-1} \Omegabf_{0}^{-1 / 2}) \otimes (\Omegabf_{0}^{-1 / 2}\left(\Ibf+w \tilde{\Deltabf}_{1}\right)^{-1} \Omegabf_{0}^{-1 / 2}). 
 \end{equation}

\noindent
This allows us to use the constraint on $\tilde{\Deltabf}_1$ to show that for an arbitrary $\kappa > 0$, $\Omegabf_{w}^{-1} \otimes \Omegabf_{w}^{-1} \succcurlyeq (1+\kappa)^{-2} \Omegabf_{0}^{-1} \otimes \Omegabf_{0}^{-1}$ in the Loewner ordering for large enough $n$.
Now, using 
\begin{equation}
 \operatorname{vec}\left(\Deltabf_{1}\right)^{\top} \left(\Omegabf_{0}^{-1} \otimes \Omegabf_{0}^{-1}\right) \operatorname{vec}\left(\Deltabf_{1}\right)=\left\|\tilde{\Deltabf}_{1}\right\|_{F}^{2}=C^2\alpha_{1}^2+C^2\beta_1^2   
\end{equation}
we are able to establish the  lower bound 
\begin{equation}
 T_4 \geqslant \frac{\left\|\tilde{\Deltabf}_{1}\right\|_{F}^{2}}{2 p\left(1+\kappa\right)^{2}}=\frac{1}{2 p\left(1+\kappa\right)^{2}}\left(C^{2} \alpha_{1}^{2}+C^{2} \beta_{1}^{2}\right) 
\end{equation}
for large enough $n$, which avoids introducing terms related to $\tau_1$. A similar comparative 
analysis holds for $T_5$ lower bound. Leveraging these lower and upper bounds, we establish a tighter inequality in (\ref{absT32}) for large enough $n$, namely,
$$
|T_{32}|\leqslant \frac{\left(1+\kappa\right)^{2}}{2}\left(T_{4}+T_{5}\right).
$$
This leaves enough portion of $T_{4}+T_{5}$ to be used for bounding the remaining terms. 

\subsubsection{Bound for \texorpdfstring{$T_{31}$}{T31}} \label{T31}

\noindent
The authors in \cite{leng2012sparse} express $T_{31}$ as 
$$
T_{31} = (pq)^{-1}\operatorname{tr}\left(\Sigmabf_{0} \Deltabf_{1}\right) \operatorname{tr}\left(\Psibf_{0} \Deltabf_{2}\right)u_n. 
$$

\noindent
Note that 
$$
u_n = \frac{1}{\operatorname{tr}\left(\Sigmabf_{0} \Deltabf_{1}\right) \operatorname{tr}\left(\Psibf_{0} \Deltabf_{2}\right)} \operatorname{tr}\left(\left(\Deltabf_{1} \otimes \Deltabf_{2}\right)\left(\Sbf-\Sigmabf_{0} \otimes \Psibf_{0}\right)\right)
$$

\noindent
depends on $\Deltabf_1$ and $\Deltabf_2$. Here $\Sbf=\left(\frac{1}{n} \sum_{i=1}^{n} \operatorname{vec}\left(\Yibf^{\top}\right) \operatorname{vec}\left(\Yibf^{\top}\right)^{\top}\right)$. The authors in \cite{leng2012sparse} observe correctly that {\it for each fixed} $\Deltabf_{1} \in \mathcal{A}_p'$ and $\Deltabf_{2} \in \mathcal{B}_q'$, the law of large numbers implies that 
\begin{equation} \label{Leng:Tang:T31:conclusion}
u_n = u_n (\Deltabf_1, \Deltabf_2) \stackrel{P_0}{\rightarrow} 0. 
\end{equation}

\noindent
However, no further analysis of the term $T_{31}$ is provided. This is a critical mistake in the proof. Even if we ignore that the entirety of $T_4+T_4$ has been used to bound $|T_{32}|$, what needs to be shown is that 
{\small
\begin{eqnarray} \label{Leng:Tang:T31:need}
& & \inf_{\Deltabf_{1} \in \mathcal{A}_p', \Deltabf_{2} \in \mathcal{B}_q'} 
\{{c}_{1,n} (T_4 + T_5) - |T_{31}|\} \nonumber\\
&=& \inf_{\Deltabf_{1} \in \mathcal{A}_p', \Deltabf_{2} \in \mathcal{B}_q'} 
\{c_{1,n} (T_4 + T_5) - |(pq)^{-1}\operatorname{tr}\left(\Sigmabf_{0} \Deltabf_{1}\right) \operatorname{tr}\left(\Psibf_{0} \Deltabf_{2}\right)u_n|\} > 0 
\end{eqnarray}
}
\noindent
with $P_0$-probability converging to $1$, where $c_{1,n}$ is an appropriate sequence of constants. This can only be achieved through a {\it uniform high-probability bound} on $u_n (\Deltabf_1, \Deltabf_2)$ over the set $\Deltabf_{1} \in \mathcal{A}_p'$ and $\Deltabf_{2} \in \mathcal{B}_q'$ is needed. 

The regularity assumptions in \cite{leng2012sparse} 
(essentially equivalent to Assumptions 1-4 in this paper) are not sufficient to provide such a uniform bound, and additional assumptions are needed. Assumption 5 along with the detailed analysis in our proof (leading to the uniform bounds in equation (\ref{woT31}) and (\ref{T31bound}))  are needed to establish the required result.

\subsection{The proof of \texorpdfstring{\cite[Theorem 3]{Yin:Li:2012}}{The proof of Theorem 3}} 

\noindent
Yin and Li \cite{Yin:Li:2012} also establish high-dimensional consistency of the SMGM-lasso estimator in Theorem 3 of that paper. Similar to the proof of \cite[Theorem 1]{leng2012sparse}, their argument aims to show that the infimum of $g$ over a relevant set is strictly larger than $g(\Omegabf_0, \Gammabf_0)$ with high probability. To achieve this, again 
the difference $g\left(\Omegabf_{0}+\Deltabf_{1}, \Gammabf_{0}+\Deltabf_{2}\right)-g\left(\Omegabf_{0}, \Gammabf_{0}\right)$ is broken up as the sum of various terms, and the goal is to leverage the positive terms to bound the others. A close examination of the proof again reveals issues with the bounding of the terms $T_{32}$ and $T_{31}$ (refered to as $K_{5}$ and $K_{6}$ respectively in \cite{Yin:Li:2012}). 

\subsubsection{Bound for \texorpdfstring{$T_{32}$}{T32}} 

\noindent
For the terms $T_4$ and $T_5$ (refered to as $K_{1}$ and $K_{2}$ respectively in \cite{Yin:Li:2012}), the following lower bounds are produced: 
\begin{equation}\label{YinT4}
T_4 \geqslant \frac{q}{2}\text{tr}(\Sigmabf_0 \Deltabf_1^{\top} \Sigmabf_0 \Deltabf_1)(1+\tilde{a}_n)
\end{equation}
\noindent
and 
\begin{equation}\label{YinT5}
T_5 \geqslant \frac{p}{2}\text{tr}(\Psibf_0 \Deltabf_2^{\top} \Psibf_0 \Deltabf_2)(1+\tilde{b}_n)
\end{equation}
where $\tilde{a}_n$ and $\tilde{b}_n$ are $o(1)$. On the other hand, the upper bound obtained for $T_{32}$ is given by 
\begin{equation}\label{YinT32}
 |T_{32}| \leqslant \frac{q}{2}\text{tr}(\Sigmabf_0 \Deltabf_1^{\top} \Sigmabf_0 \Deltabf_1)+\frac{p}{2}\text{tr}(\Psibf_0 \Deltabf_2^{\top} \Psibf_0 \Deltabf_2).
\end{equation}
Combining with (\ref{YinT4}) and (\ref{YinT5}) they conclude that $|T_{32}|$ is dominated by $T_{4}+T_{5}$ with a large probability.

This succession of arguments has the following problems:
\begin{itemize}
\item There is no guarantee that the $\tilde{a}_n$ and $\tilde{b}_n$ are positive and it is possible that the upper bound for $|T_{32}|$ is smaller than the lower bound of $T_{4}+T_{5}$. While these constants vanish in the limit, they cannot be ignored in the above bounds which need to hold {\it for every large enough $n$}. 

\item Even if $\tilde{a}_n$ and $\tilde{b}_n$ were positive, almost all of $T_4+T_5$ would be used in constraining $|T_{32}|$, leaving only $\frac{q}{2}\text{tr}(\Sigmabf_0 \Deltabf_1^{\top} \Sigmabf_0 \Deltabf_1)\tilde{a}_n+\frac{p}{2}\text{tr}(\Psibf_0 \Deltabf_2^{\top} \Psibf_0 \Deltabf_2)\tilde{b}_n \leqslant
\tilde{a}_n T_4+\tilde{b}_n T_5$ to constrain/bound the remaining terms. Following the lines of logic in \cite{Yin:Li:2012}, aside from bounding $|T_{32}|$, non-vanishing factors of $ (T_4+T_5)$ are required to bound the other terms - in particular $|T_1+T_2+T_{31}| < \tilde{c}_{1,n}(T_4+T_5+T_6)$ and $|T_7| < \tilde{c}_{2,n}(T_4+T_5)$, where $\tilde{c}_{1,n}$ and $\tilde{c}_{2,n}$ are $O(1)$. However, $\tilde{a}_n,\tilde{b}_n$ are both $o(1)$, and are inadequate for the desired upper bounds for the terms $|T_1+T_2+T_{31}|$ and $|T_7|$. 
\end{itemize}

\noindent
We refer the reader to Section \ref{T32} for a detailed description of how we we avoid these issues in our proof using sets based on the quantities $\tilde{\Deltabf}_1$ and $\tilde{\Deltabf}_2$. 

\subsubsection{Bound for \texorpdfstring{$T_{31}$}{T31}} 

The argument in the proof of \cite[Theorem 3]{Yin:Li:2012} to upper bound 
$T_{31}$ avoids the error in the proof of \cite[Theorem 1]{leng2012sparse}, and provides a uniform high-probability bound for the term $u_n (\Deltabf_1, \Deltabf_2)$ (see Section \ref{T31}) over a relevant set. However, the analysis is not tight enough, using much stronger assumptions and leading to slower/looser convergence rates than necessary. 

We first look at the assumptions in \cite{Yin:Li:2012} concerning the joint behavior of $n, p, q, s_1, s_2$. The authors in \cite{Yin:Li:2012} require that for some $k,l>1$, 
 $$q\left(p+s_{1}\right) (\log pq)^k /n= O(1), \; \; \; p\left(q+s_{2}\right) (\log pq)^l /n =O(1),$$

 \noindent
 which in particular implies (assuming $p \rightarrow \infty$ or $q \rightarrow \infty$) that 
 \begin{equation} \label{yin:li:comparison}
 q\left(p+s_{1}\right) \log pq /n\longrightarrow 0, \; \; \; p\left(q+s_{2}\right) \log pq /n \longrightarrow 0. 
 \end{equation}
 
\noindent
On the contrary, in Assumption 1 of this paper, we require
 $$\left(p+s_{1}\right) \log p /(n q) \longrightarrow 0, \; \; \; \left(q+s_{2}\right) \log q /(n p) \longrightarrow 0, 
 $$
 while in Assumption 5 (not including terms with $\lambda_1$ and $\lambda_2$), we require 
 $$\max \left(1, \frac{s_{2}}{q}, \frac{s_{1} q}{p}\right) \log (pq)/n 
\longrightarrow 0 \text{ or } \max \left(1, \frac{s_{1}}{p}, \frac{s_{2} p}{q}\right) \log(pq)/n \longrightarrow 0.$$
It can be easily seen that the constraint in (\ref{yin:li:comparison}) is stronger than both the constraints above. Note the assumptions involving $\lambda_1$ and $\lambda_2$ in this paper and for \cite[Theorem 3]{Yin:Li:2012} lead to different acceptable ranges for these two quantities, and are therefore not comparable. 

Finally, the Frobenius norm convergence rates established in \cite[Theorem 3]{Yin:Li:2012} are 
$$
\left\|\hat{\Omegabf}-\Omegabf_0\right\|_{\mathrm{F}}^2=O_{P_0}\left\{q(p+s_1)(\log p+\log q) /n \right\}
$$
and
$$
\left\|\hat{\Gammabf}-\Gammabf_0\right\|_{\mathrm{F}}^2=O_{P_0}\left\{p(q+s_2)(\log p+\log q) /n \right\}. 
$$

\noindent
These rates are significantly slower than the convergence rates 
$$
\left\|\hat{\Omegabf}-\Omegabf_0\right\|_{\mathrm{F}}^2=O_{P_0}\left\{\left(p+s_1\right) \log p /(n q)\right\}
$$
and
$$
\left\|\hat{\Gammabf}-\Gammabf_0\right\|_{\mathrm{F}}^2=O_{P_0}\left\{\left(q+s_2\right) \log q /(n p)\right\} .
$$

\noindent
established in Theorem \ref{Fnormcr}.

\section{High-dimensional convergence rates for heuristic estimator} 
\label{convergence:rate:heuristic:estimator}
As discussed in the introduction, in the current matrix-variate context, low-dimensional structures may not always be necessary for achieving consistency in high-dimensional settings. We rigorously establish this by studying the asymptotic properties of the heuristic estimators of $\Sigmabf$ and $\Psibf$ developed in \cite{SVV:2008}. The authors in \cite{SVV:2008} consider the special case when $\Psibf_{jj} = 1$ for every $1 \leqslant j \leqslant q$. In this case, each vector in the collection of $nq$ columns gathered from all the data matrices has a $p$-variate normal distribution with mean ${\underline 0}$ and covariance matrix $\Sigmabf$. While these $nq$ vectors are not all independent, \cite{SVV:2008} argue that their sample covariance matrix, given by 
\begin{equation} \label{hest}
\hat{\Sigmabf}_H := \frac{1}{nq} \sum_{i=1}^n \Yibf \Yibf^{\top}, 
\end{equation}
should be a reasonable heuristic estimator for $\Sigmabf$. To support this, they demonstrate that $\hat{\Sigmabf}_H$ is an unbiased and consistent estimator of 
$\Sigmabf$ in the classical asymptotic setting where $n \rightarrow \infty$ and $p, q$ stay fixed. The estimator $\hat{\Sigmabf}_H$ is then used to construct a similar sample covariance estimator for $\Psibf$, which is shown by \cite{SVV:2008} to be consistent for $\Psibf$. We consider a similar but slightly different heuristic estimator of $\Psibf$, given by 
\begin{equation} \label{psiest}
\hat{\Psibf}_H := \frac{1}{n \mbox{tr}(\hat{\Sigmabf}_H)} \sum_{i=1}^n 
\Yibf^{\top} \Yibf. 
\end{equation}

\noindent
We will show that under an appropriate identifiability constraint, $\hat{\Sigmabf}_H$ and $\hat{\Psibf}_H$ serve as effective estimators for $\Sigmabf$ and $\Psibf$ {\it even when the constraint $\Psibf_{jj} = 1$ for every $1 \leqslant j \leqslant q$ is removed}. We now establish their spectral norm consistency rates in a high-dimensional asymptotic regime where $p$ and/or $q$ are allowed to grow with $n$. 

We first specify the true data generating model. Under this model, for each $n$, the random matrices ${\bf Y}_{1}, {\bf Y}_{2}, \cdots, {\bf Y}_{n}$ are independent and identically distributed with a matrix normal distribution, which has mean $\mathbf{0}$, row covariance matrix $\mathbf{\Sigmabf}_0$ and column covariance matrix $\mathbf{\Psibf}_0$.
Let $\Omegabf_{0}=\Sigmabf_{0}^{-1}$ and $\Gammabf_{0}=\Psibf_{0}^{-1}$ respectively denote the row and column precision matrices. Following \cite{Theobald:Wuttke:2006}, we use the identifiability constraint $\operatorname{tr}(\Psibf_0) = q$. Note that, as in Section \ref{high:dim:consistency}, since $p$ and $q$ are allowed to change with the sample size $n$, the data matrices and the covariance/precision matrices all depend on $n$. However, we suppress their dependence on $n$ for simplicity of notation.

In order to establish our asymptotic results, we need the following mild regularity assumptions. The 
statement of each assumption below is followed by a interpretation/discussion of that assumption.

\noindent
\textit{Assumption} H1. There exists constant $\tau_{1}>0$ such that for all $n \geqslant 1$,

\begin{eqnarray*}
0<\tau_{1}<\nu_{1}\left(\Sigmabf_{0}\right) &\leqslant& \nu_{p}\left(\Sigmabf_{0}\right)<1 / \tau_{1}<\infty, \\
0<\tau_{1}<\nu_{1}\left(\Psibf_{0}\right) &\leqslant& \nu_{q}\left(\Psibf_{0}\right)<1 / \tau_{1}<\infty,
\end{eqnarray*}

\noindent where $\nu_{1}(\Abf) \leqslant \nu_{2}(\Abf) \leqslant \cdots \leqslant \nu_{m}(\Abf)$ are the eigenvalues of an $m$-dim symmetric $\Abf$.
 This standard assumption is identical to Assumption 2 in Section \ref{high:dim:consistency}, and requires the eigenvalues of the row and column covariance matrices $\Sigmabf_{0}$ and $\Psibf_{0}$ to be uniformly (in $n$) bounded away from 0 and $\infty$. 

\medskip
 
\noindent
\textit{Assumption} H2. $ \frac{\max (p,\log n)}{q}=o(n)$ and $ \frac{\max (q,\log n)}{p}=o(n)$. 
\noindent
This assumption controls the joint behavior/growth of $n, p, q$. The analogous assumption in Section \ref{high:dim:consistency} is Assumption 1. Assumption 1 is stronger than Assumption H2, which is seemingly counter-intuitive. One would think that imposing a low-dimensional sparse structure on $\Omegabf_0$ and $\Gammabf_0$ would lead to weaker requirements/assumptions for establishing consistency in Section \ref{high:dim:consistency}. However, we would like to point out that Assumption H2 will be used to obtain a {\it spectral norm} convergence rate for the heuristic estimator, whereas Assumption 1 in Section \ref{high:dim:consistency} is used to obtain a Frobenius norm convergence rate for the SMGM estimator. If one wants to establish Frobenius norm convergence rates for the heuristic estimator, the required Assumption H2 would need to be strengthened to 
$$
\frac{\max (p^2,\log n)}{q}=o(n) \mbox{ and } \frac{\max (q^2,\log n)}{p}=o(n), 
$$

\noindent
a stronger assumption than Assumption 1, as expected. 

We now state and prove our high-dimensional consistency result for the heuristic estimator $(\hat{\Sigmabf}_H, \hat{\Psibf}_H)$. Recall that for a symmetric matrix $\Abf$, we use $\|\Abf\|$ to denote the spectral norm of $\Abf$.
\begin{theorem} \label{HEmain}(Spectral norm convergence rates for heuristic estimators) Under Assumptions H1-H2, as $n\rightarrow\infty$, for arbitrarily chosen constant $\tilde{C}_{0}$, heuristic covariance estimators $\hat{\Sigmabf}_H$ and $\hat{\Psibf}_{H}$ proposed in (\ref{hest}) and (\ref{psiest}) 
satisfy 
\begin{eqnarray*}
\|\hat{\Sigmabf}_H-\Sigmabf_0\| &=& O_{P_0} \left( \sqrt{\frac{\max (p,\log n)}{nq}} \right) \mbox{ and }\\
\|\hat{\Psibf}_{H}-\Psibf_{0}\| &=& O_{P_0} \left( 
\sqrt{\max\left(\frac{\max(p,\log n)}{nq},\frac{\max(q,\log n)}{np}\right)} \right). 
\end{eqnarray*}

\noindent
The corresponding precision matrix estimators $\hat{\Omegabf}_H = \hat{\Sigmabf}_H^{-1}$ and $\hat{\Gammabf}_H = \hat{\Psibf}_H^{-1}$ also satisfy 
 \begin{eqnarray*}
\|\hat{\Omegabf}_H-\Omegabf_0\| &=& O_{P_0} \left( \sqrt{\frac{\max (p,\log n)}{nq}} \right) \mbox{ and } \\
\|\hat{\Gammabf}_{H}-\Gammabf_{0}\| &=& O_{P_0} \left( 
\sqrt{\max\left(\frac{\max(p,\log n)}{nq},\frac{\max(q,\log n)}{np}\right)} \right). 
\end{eqnarray*}
\end{theorem}
The proof of Theorem \ref{HEmain} is provided in 
the Appendix. 

\medskip

\noindent
{\bf Acknowledgements}. Khare's work on this paper was partially supported by NSF-DMS-2410667. 

\bibliographystyle{imsart-nameyear}
\bibliography{refs.bib}

@article{bickel2008regularized,
  title   = {Regularized estimation of large covariance matrices},
  author  = {Bickel, Peter J and Levina, Elizaveta},
  journal = {Annals of Statistics},
  year    = {2008}
}

@article{khare2015convex,
  title   = {A convex pseudolikelihood framework for high dimensional partial correlation estimation with convergence guarantees},
  author  = {Khare, Kshitij and Oh, Sang-Yun and Rajaratnam, Bala},
  journal = {JRSSB},
  volume  = {77},
  number  = {4},
  pages   = {803--825},
  year    = {2015}
}

@article{peng2009partial,
  title   = {Partial correlation estimation by joint sparse regression models},
  author  = {Peng, Jie and Wang, Pei and Zhou, Nengfeng and Zhu, Ji},
  journal = {Journal of the American Statistical Association},
  volume  = {104},
  number  = {486},
  pages   = {735--746},
  year    = {2009}
}

@article{rothman2008sparse,
  title   = {Sparse permutation invariant covariance estimation},
  author  = {Rothman, Adam J and Bickel, Peter J and Levina, Elizaveta and Zhu, Ji},
  journal = {Electronic Journal of Statistics},
  volume  = {2},
  pages   = {494--515},
  year    = {2008}
}

@article{rudelson2013hanson,
  title   = {Hanson-Wright inequality and sub-Gaussian concentration},
  author  = {Rudelson, Mark and Vershynin, Roman},
  journal = {Electronic Communications in Probability},
  volume  = {18},
  pages   = {1--9},
  year    = {2013}
}

@article{THZ:2013,
  author  = {Tsiligkaridis, Theodoros and Hero III, Alfred O. and Zhou, Shuheng},
  title   = {On Convergence of Kronecker Graphical Lasso Algorithms},
  journal = {IEEE Transactions on Signal Processing},
  year    = {2013},
  volume  = {61},
  number  = {7},
  pages   = {1743--1755}
}

@article{franks2021near,
  title   = {Near optimal sample complexity for matrix and tensor normal models via geodesic convexity},
  author  = {Franks, Cole and Oliveira, Rafael and Ramachandran, Akshay and Walter, Michael},
  journal = {arXiv preprint},
  year    = {2021}
}

@article{leng2012sparse,
  title   = {Sparse matrix graphical models},
  author  = {Leng, Chenlei and Tang, Cheng Yong},
  journal = {Journal of the American Statistical Association},
  volume  = {107},
  number  = {499},
  pages   = {1187--1200},
  year    = {2012}
}

@article{banerjee2014posterior,
  title   = {Posterior convergence rates for estimating large precision matrices using graphical models},
  author  = {Banerjee, Sayantan and Ghosal, Subhashis},
  journal = {Electronic Journal of Statistics},
  year    = {2014}
}

@article{Dutilleul:1999,
  author  = {Dutilleul, Pierre},
  title   = {The MLE algorithm for the matrix normal distribution},
  journal = {Journal of Statistical Computation and Simulation},
  volume  = {64},
  number  = {2},
  pages   = {105--123},
  year    = {1999}
}

@article{banerjee2015bayesian,
  title   = {Bayesian structure learning in graphical models},
  author  = {Banerjee, Sayantan and Ghosal, Subhashis},
  journal = {Journal of Multivariate Analysis},
  volume  = {136},
  pages   = {147--162},
  year    = {2015}
}

@article{WJS:2008,
  author  = {Werner, Karl and Jansson, Magnus and Stoica, Petre},
  title   = {On Estimation of Covariance Matrices With Kronecker Product Structure},
  journal = {IEEE Transactions on Signal Processing},
  year    = {2008},
  volume  = {56},
  number  = {2},
  pages   = {478--491}
}

@article{Lu:Zimmerman:2005,
  author  = {Lu, Nelson and Zimmerman, Dale L.},
  title   = {The likelihood ratio test for a separable covariance matrix},
  journal = {Statistics \& Probability Letters},
  volume  = {73},
  number  = {4},
  pages   = {449--457},
  year    = {2005}
}

@article{Yin:Li:2012,
  author  = {Yin, Jianxin and Li, Hongzhe},
  title   = {Model selection and estimation in the matrix normal graphical model},
  journal = {Journal of Multivariate Analysis},
  volume  = {107},
  pages   = {119--140},
  year    = {2012}
}

@article{SVV:2008,
  title   = {Models with a Kronecker product covariance structure: estimation and testing},
  author  = {Srivastava, Muni S and von Rosen, Tatjana and von Rosen, Dietrich},
  journal = {Mathematical Methods of Statistics},
  volume  = {17},
  pages   = {357--370},
  year    = {2008}
}

@article{Galecki:1994,
  author  = {Galecki, Andrzej T.},
  title   = {General class of covariance structures for two or more repeated factors in longitudinal data analysis},
  journal = {Communications in Statistics - Theory and Methods},
  volume  = {23},
  number  = {11},
  pages   = {3105--3119},
  year    = {1994}
}

@article{Naik:Rao:2001,
  author  = {Naik, Dayanand N. and Rao, Shantha S.},
  title   = {Analysis of multivariate repeated measures data with a Kronecker product structured covariance matrix},
  journal = {Journal of Applied Statistics},
  volume  = {28},
  number  = {1},
  pages   = {91--105},
  year    = {2001}
}

@article{Wang:West:2009,
  author  = {Wang, Hao and West, Mike},
  title   = {Bayesian analysis of matrix normal graphical models},
  journal = {Biometrika},
  volume  = {96},
  number  = {4},
  pages   = {821--834},
  year    = {2009}
}

@article{Zhou:2014,
  author  = {Zhou, Shuheng},
  title   = {Gemini: Graph estimation with matrix variate normal instances},
  journal = {The Annals of Statistics},
  volume  = {42},
  number  = {2},
  pages   = {532--562},
  year    = {2014}
}

@article{Yuan:Lin:2007,
  author  = {Yuan, Ming and Lin, Yi},
  title   = {Model Selection and Estimation in the Gaussian Graphical Model},
  journal = {Biometrika},
  volume  = {94},
  number  = {1},
  pages   = {19--35},
  year    = {2007}
}

@article{10.3150/17-BEJ980,
  author  = {Leng, Chenlei and Pan, Guangming},
  title   = {Covariance estimation via sparse Kronecker structures},
  journal = {Bernoulli},
  volume  = {24},
  number  = {4B},
  pages   = {3833--3863},
  year    = {2018}
}

@book{Gupta:Nagar:2000,
  author    = {Gupta, A. K. and Nagar, D. K.},
  title     = {Matrix Variate Distributions},
  publisher = {Chapman \& Hall/CRC},
  address   = {Boca Raton},
  year      = {2000}
}

@article{Dawid:1981,
  author  = {Dawid, A. P.},
  title   = {Some Matrix-Variate Distribution Theory: Notational Considerations and a Bayesian Application},
  journal = {Biometrika},
  volume  = {68},
  number  = {1},
  pages   = {265--274},
  year    = {1981}
}

@article{Theobald:Wuttke:2006,
  author  = {Theobald, Douglas L. and Wuttke, Deborah S.},
  title   = {Empirical Bayes hierarchical models for regularizing maximum likelihood estimation in the matrix Gaussian Procrustes problem},
  journal = {PNAS},
  volume  = {103},
  number  = {49},
  pages   = {18521--18527},
  year    = {2006}
}

@book{Lauritzen:1996,
  author    = {Lauritzen, Steffen L.},
  title     = {Graphical Models},
  publisher = {Oxford University Press},
  year      = {1996}
}

@inbook{Vershynin_2012,
  author    = {Vershynin, Roman},
  title     = {Introduction to the non-asymptotic analysis of random matrices},
  booktitle = {Compressed Sensing: Theory and Applications},
  publisher = {Cambridge University Press},
  address   = {Cambridge},
  year      = {2012},
  pages     = {210--268}
}

@article{GKM:2019,
  author  = {Ghosh, Satyajit and Khare, Kshitij and Michailidis, George},
  title   = {High-Dimensional Posterior Consistency in Bayesian Vector Autoregressive Models},
  journal = {Journal of the American Statistical Association},
  volume  = {114},
  number  = {526},
  pages   = {735--748},
  year    = {2019}
}

@article{gtex2017nature,
  title        = {Genetic effects on gene expression across human tissues},
  author       = {{GTEx Consortium}},
  journal      = {Nature},
  year         = {2017},
  volume       = {550},
  number       = {7675},
  pages        = {204--213},
  doi          = {10.1038/nature24277},
  url          = {https://www.nature.com/articles/nature24277.pdf},
  note         = {Analyzes gene expression across 44 human tissues}
}

@article{gtex2015science,
  title        = {The Genotype-Tissue Expression (GTEx) pilot analysis: Multitissue gene regulation in humans},
  author       = {{GTEx Consortium}},
  journal      = {Science},
  year         = {2015},
  volume       = {348},
  number       = {6235},
  pages        = {648--660},
  doi          = {10.1126/science.1262110},
  url          = {https://www.science.org/doi/reader/10.1126/science.1262110},
  note         = {Pilot RNA-seq across 43 tissues from 175 individuals}
}

@misc{physionet_eegmmidb,
  title        = {EEG Motor Movement/Imagery Dataset (EEGMMI DB)},
  author       = {Schalk, Gerwin and McFarland, Dennis J. and Hinterberger, Thilo and Birbaumer, Niels and Wolpaw, Jonathan R. and Wadsworth BCI2000 Team and PhysioNet},
  howpublished = {PhysioNet},
  year         = {2009},
  url          = {https://physionet.org/content/eegmmidb/1.0.0/},
  note         = {64-channel EEG from 109 subjects performing motor execution and imagery tasks; include the standard PhysioNet citation: Goldberger AL \textit{et al.}, \emph{Circulation} 101(23):e215--e220, 2000}
}

\begin{appendix}
\section*{Appendix - Proof of Theorem 4.1}\label{appn}

We first show that $\hat{\Sigmabf}_H$ is unbiased for $\Sigmabf_0$. Note
that 
$$
E_{P_0} \left[{\bf Y}_{1} {\bf Y}_{1}^{\top}\right]=\sum_{j=1}^{q} E_{P_0} \left[ ({\bf Y}_{1})_{:j} ({\bf Y}_{1})_{:j}^{\top} \right],
$$

\noindent
where $({\bf Y}_{1})_{:j}$ denotes the $j^{th}$ column of ${\bf Y}_1$. 
Note that for any $1 \leqslant k \leqslant p$ and $1 \leqslant l \leqslant p$, the $(k,l)^{th}$ 
entry of $E_{P_0} \left[ ({\bf Y}_{1})_{:j} ({\bf Y}_{1})_{:j}^{\top} \right]$ 
is given by
\begin{eqnarray*}
E_{P_0} \left[\left(Y_{1}\right)_{k j}\left(Y_{1}\right)_{l j}\right]=\left(\Sigma_{0}\right)_{k l}\left(\Psi_{0}\right)_{j j}. 
\end{eqnarray*}

\noindent
It follows that 
$$
E_{P_0} \left[ {\bf Y}_{1} {\bf Y}_{1}^{\top} \right] = \left( \sum_{j=1}^{q} 
\left(\Psi_{0}\right)_{jj} \right) \Sigmabf_{0} = \operatorname{tr} 
\left(\Psibf_{0}\right) \Sigmabf_{0} = q \Sigmabf_{0} 
$$ 

\noindent
based on the identifiability constraint. Since ${\bf Y}_{1}, {\bf Y}_{2}, \cdots, {\bf Y}_{n}$ are identically distributed, we obtain 
\begin{equation}
\label{ESigmaH}
E_{P_0} \left[ \hat{\Sigmabf}_{H} \right] = E_{P_0} \left[ \frac{1}{n q} \sum_{i=1}^{n} 
\Yibf \Yibf^{\top} \right] = \Sigmabf_{0}.   
\end{equation}

\noindent We can similarly obtain 
\begin{equation}
\label{EPsiH}
 E_{P_0}\left[\frac{1}{n \operatorname{tr}\left(\Sigmabf_{0}\right)} \sum_{i=1}^{n} \Yibf^{\top} \Yibf\right]=\Psibf_{0}.   
\end{equation}
Equations (\ref{ESigmaH}) and (\ref{EPsiH}) explain why the estimator $\hat{\Psibf}_H$ in eq. (4) in the main paper, though not unbiased, is a reasonable estimator of $\Psibf_0$.

We now proceed to show $\hat{\Sigmabf}_{H}$ and $\hat{\Psibf}_{H}$ are consistent in spectral norm and obtain asymptotic high-dimensional spectral norm convergence rates for both these heuristic estimators.
Note that 
\begin{equation} \label{sup:norm}
\|\hat{\Sigmabf}_H - \Sigmabf_0\| = \sup _{\|\underline{x}\|_{2} = 1}\left|\underline{x}^{\top} \hat{\Sigmabf}_{H} \underline{x}-\underline{x}^{\top} \Sigmabf_{0} \underline{x} \right|
\end{equation}

\noindent
Given this, our first objective is to provide a high probability bound for the term $\left|\underline{x}^{\top} \hat{\Sigmabf}_{H} \underline{x}-\underline{x}^{\top} \Sigmabf_{0} \underline{x} \right|$ for an arbitrary ${\underline x} \in \mathbb{R}^p$ with $\|\underline{x}\|_{2} = 1$. Let $\tXibf:=\Sigmabf_{0}^{-1 / 2} \Yibf \Psibf_{0}^{-1 / 2}$ for $1 \leqslant i \leqslant n$. It follows that $\tilde{{\bf X}}_1, \tilde{{\bf X}}_2, \cdots, \tilde{{\bf X}}_n$ are i.i.d. from a matrix-variate normal distribution with mean matrix ${\bf 0}$, row covariance matrix ${\bf I}_p$, and column covariance matrix ${\bf I}_q$. Then, for an arbitrary ${\underline x} \in \mathbb{R}^p$ with $\|\underline{x}\|_{2} = 1$, properties of the trace operator and the vectorization operator imply that 
\begin{eqnarray}
\notag
\underline{x}^{\top} \hat{\Sigmabf}_{H} \underline{x}&=&\frac{1}{n q} \sum_{i=1}^{n} \underline{x}^{\top} \Yibf \Yibf^{\top} \underline{x} \\
\notag
& =&\frac{1}{n q} \sum_{i=1}^{n} \underline{x}^{\top} \Sigmabf_{0}^{1 / 2} \tXibf \Psibf_{0} \tXibf^{\top} \Sigmabf_{0}^{1 / 2} \underline{x} \\
\notag
& =&\frac{1}{n q} \sum_{i=1}^{n}  \text { tr } \left(\underline{x}^{\top} \Sigmabf_{0}^{1 / 2} \tXibf \Psibf_{0} \tXibf^{\top} \Sigmabf_{0}^{1 / 2} \underline{x}\right) \\
\notag
& =&\frac{1}{n q} \sum_{i=1}^{n} \text { tr }\left(\tXibf^{\top}\left(\Sigmabf_{0}^{1 / 2} \underline{x} \underline{x}^{\top} \Sigmabf_{0}^{1 / 2}\right) \tXibf \Psibf_{0}\right) \\
\notag
& =&\frac{1}{n q} \sum_{i=1}^{n} \operatorname{vec}\left(\tXibf\right)^{\top} \operatorname{vec}\left(\left(\Sigmabf_{0}^{1 / 2} \underline{x} \underline{x}^{\top} \Sigmabf_{0}^{1 / 2}\right) \tXibf \Psibf_{0}\right)\\
\label{MPP}
& =&\frac{1}{n q} \sum_{i=1}^{n} \operatorname{vec}\left(\tXibf\right)^{\top}\left(\Psibf_{0} \otimes\left(\Sigmabf_{0}^{1 / 2} \underline{x} \underline{x}^{\top} \Sigmabf_{0}^{1 / 2}\right)\right) \operatorname{vec}\left(\tXibf\right) \\
\label{G1}
& =:& \underline{z}^{\top} \Gbf_{1} \underline{z},
\end{eqnarray}
where equation (\ref{MPP}) follows from the mixed-product property of the Kronecker product, 
\begin{eqnarray*}
\underline{z}:=\left(
\operatorname{vec}\left(\tilde{\Xbf}_{1}\right)^{\top} \;
\operatorname{vec}\left(\tilde{\Xbf}_{2}\right)^{\top} \;
\cdots \; \operatorname{vec}\left(\tilde{\Xbf}_{n}\right)^{\top}\right)^{\top} \in 
\mathbb{R}^{npq}, 
\end{eqnarray*}
\noindent
and 
\begin{eqnarray*}
\Gbf_{1}:=\frac{1}{n q} B D_{n}\left(\Psibf_{0} \otimes\left(\Sigmabf_{0}^{1 / 2} \underline{x} \underline{x}^{\top} \Sigmabf_{0}^{1 / 2}\right)\right). 
\end{eqnarray*}

\noindent
Note that $B D_{n}\left(\Psibf_{0} \otimes\left(\Sigmabf_{0}^{1 / 2} \underline{x} \underline{x}^{\top} \Sigmabf_{0}^{1 / 2}\right)\right)$ denotes a block diagonal matrix with $n$ blocks along its main diagonal, with each diagonal block equal to $\Psibf_{0} \otimes\left(\Sigmabf_{0}^{1 / 2} \underline{x} \underline{x}^{\top} \Sigmabf_{0}^{1 / 2}\right)$. Note further that 
\begin{eqnarray}
\notag
\|\Gbf_{1}\|_{F}^{2} & =&\frac{n}{n^{2} q^{2}}\left\|\Psibf_{0} \otimes\left(\Sigmabf_{0}^{1 / 2} \underline{x} \underline{x}^{\top} \Sigmabf_{0}^{1 / 2}\right)\right\|_{F}^{2} \\
\label{G1Feq1}
& =&\frac{n}{n^{2} q^{2}}\left\|\Psibf_{0}\right\|_{F}^{2}\left\|\left(\Sigmabf_{0}^{1 / 2} \underline{x} \underline{x}^{\top} \Sigmabf_{0}^{1 / 2}\right)\right\|_{F}^{2}\\
\label{G1Feq2}
& =&\frac{n}{n^{2} q^{2}}\left\|\Psibf_{0}\right\|_{F}^{2}\left(\underline{x}^{\top} \Sigmabf_{0} x\right)^{2} \\
\label{G1Feq3}
 & \leqslant& \frac{1}{n q}\left\|\Psibf_{0}\right\|^{2}\left(\underline{x}^{\top} \Sigmabf_{0} \underline{x}\right)^{2},
\end{eqnarray}
\begin{eqnarray}
\label{G1Seq1}
\|\Gbf_{1}\| &=& \frac{1}{n q}\left\|\Psibf_{0}\right\|\left\|\Sigmabf_{0}^{1 / 2} \underline{x} \underline{x}^{\top} \Sigmabf_{0}^{1 / 2}\right\|\\
\label{G1Seq2}
&\leqslant& \frac{1}{n q}\left\|\Psibf_{0}\right\|\left(\underline{x}^{\top} \Sigmabf_{0} \underline{x}\right),    
\end{eqnarray}
where equality (\ref{G1Feq1}) follows from the property of the Frobenius norm with respect to the Kronecker product; equality (\ref{G1Feq2}) follows from the fact that for any vector $u$, $\left\|\underline{u} \underline{u}^{\top}\right\|_{F}=\underline{u}^{\top} \underline{u}$, with $u=\Sigmabf_{0}^{1 / 2} \underline{x}$;
equality (\ref{G1Seq1}) follows from spectral norm multiplicativity of the Kronecker product; inequalities (\ref{G1Feq3}) and (\ref{G1Seq2}) follows from the inequality relating the Frobenius norm and the spectral norm. Combining (\ref{G1}), (\ref{G1Feq3}) and (\ref{G1Seq2}) with the fact that $E_{P_0}\left[\underline{x}^{\top} \hat{\Sigmabf}_{H} \underline{x}\right]=\underline{x}^{\top} \Sigmabf_{0} \underline{x}$ (which follows from (\ref{ESigmaH})) and the Hanson-Wright inequality \cite[Theorem 1.1]{rudelson2013hanson}, for every $t_1 > 0$ we obtain 
\begin{eqnarray}
\notag
& & P_{0}\left(\left|\underline{x}^{\top} \hat{\Sigmabf}_{H} \underline{x}-\underline{x}^{\top} \Sigmabf_{0} \underline{x}\right|>t_{1}\right)\\
\notag
&=& P_{0}\left(\left|\underline{z}^{\top} \Gbf_{1} \underline{z}-E_{P_0}(\underline{z}^{\top} \Gbf_{1}\underline{z})\right|>t_{1}\right)\\
\notag
&\leqslant& 2 \exp{\left(-c \min \left(\frac{t_{1}^{2}}{4\|\Gbf_{1}\|_{F}^{2}}, \frac{t_{1}}{2\|\Gbf_{1}\|}\right)\right)}\\
\notag
&\leqslant& 2 \exp \left(-c \min \left(\frac{n q t_{1}^{2}}{4\left\|\Psibf_{0}\right\|^{2}\left(\underline{x}^{\top} \Sigmabf_{0} \underline{x}\right)^{2}}, \frac{n q t_{1}}{2\left\|\Psibf_{0}\right\|\left(\underline{x}^{\top} \Sigmabf_{0} \underline{x}\right)}\right) \right)\\
\label{P0teq1}
&=&2 \exp \left(-cnq \min \left(\left(\frac{t_{1}}{2\left\|\Psibf_{0}\right\|\left(\underline{x}^{\top} \Sigmabf_{0} \underline{x}\right)}\right)^{2}, \frac{t_{1}}{2\left\|\Psibf_{0}\right\|\left(\underline{x}^{\top} \Sigmabf_{0} \underline{x}\right)}\right)\right).
\end{eqnarray}

\noindent
Here $c$ is a universal fixed constant. It follows from Assumption H1 along with $\|\underline{x}\|_{2} = 1$ that $\tau_{1} <\underline{x}^{\top} \Sigmabf_{0} \underline{x}<\frac{1}{\tau_{1}}$ and $ \tau_{1}<\left\|\Psibf_{0}\right\| <\frac{1}{\tau_{1}}$. In particular, this implies that 
$$
\frac{\tau_{1}^{2}t_{1}}{2 }<\frac{t_{1}}{2\left\|\Psibf_{0}\right\|\left(\underline{x}^{\top} \Sigmabf_{0} \underline{x}\right)}<\frac{t_{1}}{2 \tau_{1}^{2}}<1
$$

\noindent
whenever $0< t_1 < 2 \tau_1^2$. Using (\ref{P0teq1}) along with the above observations, for any $0<t_1 < 2 \tau_1^2$, we obtain 
{\small
\begin{eqnarray} 
\label{P0teq1.2}
 P_{0}\left(\left|\underline{x}^{\top} \hat{\Sigmabf}_{H} \underline{x}-\underline{x}^{\top} \Sigmabf_{0} \underline{x}\right|> t_{1} \right)
\leqslant 2 \exp \left( \frac{-c n q} {4\left\|\Psibf_{0}\right\|^{2}\left(\underline{x}^{\top} \Sigmabf_{0} \underline{x}\right)^{2}} t_{1}^{2}\right)  \leqslant 2 \exp\left(\frac{-c \tau_{1}^{4} n q}{4} t_{1}^{2}\right) 
\end{eqnarray}}
For any constant $\tilde{C}_{0} > 0$, since $\tilde{C}_{0} \sqrt{\frac{\max(p,\log n)}{nq}} \rightarrow 0$ as $n \rightarrow \infty$
by Assumption H2, it follows that  $\tilde{C}_{0} \sqrt{\frac{\max(p,\log n)}{nq}} < 2 
\tau_1^2$ eventually. Hence, let $t_1=\tilde{C}_{0} \sqrt{\frac{\max(p,\log n)}{nq}}$ in (\ref{P0teq1.2}), we get

{\small
\begin{eqnarray} \label{P0teq2}
 P_{0}\left(\left|\underline{x}^{\top} \hat{\Sigmabf}_{H} \underline{x}-\underline{x}^{\top} \Sigmabf_{0} \underline{x}\right|> \tilde{C}_{0} \sqrt{\frac{\max(p,\log n)}{nq}} \right)
\leqslant 2 \exp \left( \frac{-c \tau_1^4 \tilde{C}_{0}^2 \max(p, \log n)}{4} \right) 
\end{eqnarray}}
\noindent
for large enough $n$.

Note that the bound in (\ref{P0teq2}) holds for every 
$\underline{x}$ such that $\|\underline{x}\|_{2} = 1$. Using (\ref{sup:norm}), (\ref{P0teq2}) along with the covering argument in \cite[Lemma 5.2]{Vershynin_2012} and \cite[Lemma B.2]{GKM:2019}, it follows that 
\begin{eqnarray}
\notag
& & P_{0} \left( ||\hat{\Sigmabf}_{H}-\Sigmabf_{0}|| > \tilde{C}_{0} \sqrt{\frac{\max(p,\log n)}{nq}} \right)\\
\notag
&=& P_{0}\left(\sup _{\|\underline{x}\|_{2} = 1}\left|\underline{x}^{\top} \hat{\Sigmabf}_{H} \underline{x}-\underline{x}^{\top} \Sigmabf_{0} \underline{x}\right| > \tilde{C}_{0} \sqrt{\frac{\max(p,\log n)}{nq}} \right)\\
\notag
&\leqslant& 21^p \sup _{\|\underline{x}\|_{2} = 1} P\left(\left|\underline{x}^{\top} \hat{\Sigmabf}_{H} \underline{x}-\underline{x}^{\top} \Sigmabf_{0} \underline{x}\right| > \tilde{C}_{0} \sqrt{\frac{\max(p,\log n)}{nq}} \right)\\
\label{sigmacovrate}
&\leqslant& 2\exp \left( -\frac{c \tau_1^4 \tilde{C}_{0}^2 \max(p, \log n)}{4} + \log 21\cdot p \right) \rightarrow 0 
\end{eqnarray}

\noindent
for any $\tilde{C}_{0}>\frac{2}{\tau_1^2}\sqrt{\frac{\log 21}{c}}$ as $n \rightarrow \infty$. Thus, the required spectral norm convergence rate for $\hat{\Sigmabf}_H$ has been established. 

Next, we proceed to the analysis of $\hat{\Psibf}_{H}$. Note that 
\begin{eqnarray*}
\hat{\Psibf}_{H}=\frac{1}{n \operatorname{tr}\left(\hat{\Sigmabf}_{H}\right)} \sum_{i=1}^{n} \Yibf^{\top} \Yibf=\frac{p}{\operatorname{tr}\left(\hat{\Sigmabf}_{H}\right)} \cdot \frac{1}{n p} \sum_{i=1}^{n} \Yibf^{\top} \Yibf=:T_{H} \cdot \tilde{\Psibf}_{H},
\end{eqnarray*}
where 
$$
T_{H}:=\frac{p}{\operatorname{tr}\left(\hat{\Sigmabf}_{H}\right)} \mbox{ and } \tilde{\Psibf}_{H}:=\frac{1}{n p} \sum_{i=1}^{n} \Yibf^{\top} \Yibf. 
$$

\noindent
Additionally, let $T_{0}:=\frac{p}{\operatorname{tr}\left(\Sigmabf_{0}\right)}$ and $\tilde{\Psibf}_{0}:=\frac{\operatorname{tr}\left(\Sigmabf_{0}\right)}{p} \Psibf_{0}.$ We first prove $\tilde{\Psibf}_{H}$ is consistent for $\tilde{\Psibf}_{0}$ and find its asymptotic high-dimensional spectral norm convergence rate following a similar approach as the one we used for $\hat{\Sigmabf}_{H}$ above. Since 
\begin{eqnarray}
\label{sup:norm2}
||\tilde{\Psibf}_{H}-\tilde{\Psibf}_{0}||=
\sup _{\|\underline{y}\|_{2} = 1}\left|\underline{y}^{\top} \tilde{\Psibf}_{H} \underline{y}-\underline{y}^{\top} \tilde{\Psibf}_{0} \underline{y}\right|,    
\end{eqnarray}
our first objective is to provide a high probability bound for the term\\
$\left|\underline{y}^{\top} \tilde{\Psibf}_{H} \underline{y}-\underline{y}^{\top} \tilde{\Psibf}_{0} \underline{y}\right| $ for an arbitrary ${\underline y} \in \mathbb{R}^q$ with $\|\underline{y}\|_{2} = 1$. By employing a methodology analogous to that used to express $\underline{x}^{\top} \hat{\Sigmabf}_{H} \underline{x}$ as $\underline{z}^{\top} \Gbf_{1} \underline{z}$ in (\ref{G1}), for an arbitrary ${\underline y} \in \mathbb{R}^q$ with $\|\underline{y}\|_{2} = 1$, we can obtain
\begin{eqnarray}
\label{G2}
\underline{y}^{\top} \tilde{\Psibf}_{H} \underline{y}
 =: \underline{z}^{\top} \Gbf_{2} \underline{z},
\end{eqnarray}
where
\begin{eqnarray*}
\underline{z}=\left(
\operatorname{vec}\left(\tilde{\Xbf}_{1}\right)^{\top} \;
\operatorname{vec}\left(\tilde{\Xbf}_{2}\right)^{\top} \;
\cdots  \; \operatorname{vec}\left(\tilde{\Xbf}_{n}\right)^{\top}\right)^{\top} \in 
\mathbb{R}^{npq},
\end{eqnarray*}
\begin{eqnarray*}
\tXibf=\Sigmabf_{0}^{-1 / 2} \Yibf \Psibf_{0}^{-1 / 2} \text{ for } 1 \leqslant i \leqslant n
\end{eqnarray*}
and
\begin{equation*}
\Gbf_{2}:=\frac{1}{n p} B D_{n}\left(\left(\Psibf_{0}^{1 / 2} \underline{y} \underline{y}^{\top} \Psibf_{0}^{1 / 2}\right) \otimes \Sigmabf_{0}\right).    
\end{equation*}
Note here entries of $\underline{z}$ are iid $N(0,1)$. Further, following a path very similar to the one which led to upper bounds for $\|\Gbf_{1}\|_{F}^{2}$ and $\|\Gbf_{1}\|$ in (\ref{G1Feq3}) and (\ref{G1Seq2}), we can get the following upper bounds for $\|\Gbf_{2}\|_{F}^{2}$ and $\|\Gbf_{2}\|$: 
\begin{equation} \label{G2Feq1}
\|\Gbf_{2}\|_{F}^{2} \leqslant \frac{1}{n p}\|\Sigmabf_0\|^{2}\left(\underline{y}^{\top} \Psibf_{0} \underline{y}\right)^{2} \mbox{ and } \|\Gbf_{2}\| \leqslant \frac{1}{n p}\left\|\Sigmabf_{0}\right\|\left(\underline{y}^{\top} \Psibf_{0} \underline{y}\right).
\end{equation}
Next, combining (\ref{G2}) and (\ref{G2Feq1}) with the fact that $E_{P_0}\left[\underline{y}^{\top} \tilde{\Psibf}_{H} \underline{y}\right]=\underline{y}^{\top} \tilde{\Psibf}_{0} \underline{y}$ (which follows from (\ref{EPsiH})) and the Hanson-Wright inequality \cite[Theorem 1.1]{rudelson2013hanson}, for every $t_2 > 0$ we obtain 

\begin{eqnarray}
\notag
&&P_{0}\left(\left|\underline{y}^{\top} \tilde{\Psibf}_{H} \underline{y}-\underline{y}^{\top} \tilde{\Psibf}_{0} \underline{y}\right|>t_{2}\right) \\
\label{P0t2eq2}
&\leqslant&2 \exp \left(-cnp \min \left(\left(\frac{t_{2}}{2\left\|\Sigmabf_{0}\right\|\left(\underline{y}^{\top} \Psibf_{0} \underline{y}\right)}\right)^{2}, \frac{t_{2}}{2\left\|\Sigmabf_{0}\right\|\left(\underline{y}^{\top} \Psibf_{0} \underline{y}\right)}\right)\right).
\end{eqnarray}
It follows from Assumption H1 along with $\|\underline{y}\|_{2} = 1$ that $\tau_{1} <\underline{y}^{\top} \Psibf_{0} \underline{y}<\frac{1}{\tau_{1}}$ and $ \tau_{1}<\left\|\Sigmabf_{0}\right\| <\frac{1}{\tau_{1}}$. In particular, this implies 
$$
\frac{\tau_{1}^{2}t_{2}}{2 }<\frac{t_{2}}{2\left\|\Sigmabf_{0}\right\|\left(\underline{y}^{\top} \Psibf_{0} \underline{y}\right)}<\frac{t_{2}}{2 \tau_{1}^{2}}<1,
$$

\noindent
whenever $0<t_2 < 2 \tau_1^2$. For any constant $\tilde{C}_{0} > 0$, note that $\tilde{C}_{0} \sqrt{\frac{\max(q,\log n)}{np}} \rightarrow 0$ as $n \rightarrow \infty$ by Assumption H2. Using (\ref{P0t2eq2}) along with the above observations, and arguments similar to those in (\ref{P0teq1.2}) and (\ref{P0teq2}), we obtain 
{\small
\begin{eqnarray}\label{P0teq3}
P_{0}\left(\left|\underline{y}^{\top} \tilde{\Psibf}_{H} \underline{y}-\underline{y}^{\top} \tilde{\Psibf}_{0} \underline{y}\right|> \tilde{C}_{0} \sqrt{\frac{\max(q,\log n)}{np}} \right)
\leqslant 2 \exp \left( \frac{-c \tau_1^4 \tilde{C}_{0}^2 \max(q, \log n)}{4} \right) 
\end{eqnarray}}
for every $n$ large enough such that $\tilde{C}_{0} \sqrt{\frac{\max(q,\log n)}{np}} < 2 
\tau_1^2$. Note that the bound in (\ref{P0teq3}) holds for every 
$\underline{y}$ such that $\|\underline{y}\|_{2} = 1$. Using (\ref{sup:norm2}), (\ref{P0teq3}) along with the covering argument in \cite[Lemma 5.2]{Vershynin_2012} and \cite[Lemma B.2]{GKM:2019}, following a similar line of arguments as in (\ref{sigmacovrate}), we obtain 
\begin{eqnarray}
\notag
& & P_{0} \left(\left\|\tilde{\Psibf}_{H}-\tilde{\Psibf}_{0}\right\| > \tilde{C}_{0} \sqrt{\frac{\max(q,\log n)}{np}} \right)\\
\label{tildepsicovrate}
&\leqslant& 2\exp \left( -\frac{c \tau_1^4 \tilde{C}_{0}^2 \max(q, \log n)}{4} + \log 21\cdot q \right) \rightarrow 0 
\end{eqnarray}

\noindent
for any $\tilde{C}_{0}>\frac{2}{\tau_1^2}\sqrt{\frac{\log 21}{c}}$ as $n \rightarrow \infty$. With the asymptotic high-dimensional spectral norm convergence rate of $\tilde{\Psibf}_{H}$ in hand, we note that for any $t_3>0$,
\begin{eqnarray}
\notag
&& P_{0}\left(\left\|\hat{\Psibf}_{H}-\Psibf_{0}\right\|>t_{3}\right)\\
\notag
&=&P_{0}\left(\left\|T_{H} \tilde{\Psibf}_{H}-T_{0} \tilde{\Psibf}_{0}\right\|>t_{3}\right) \\
\notag
& =&P_{0}\left(\left\|T_{H}\left(\tilde{\Psibf}_{H}-\tilde{\Psibf}_{0}\right)+\left(T_{H}-T_{0}\right) \tilde{\Psibf}_{0}\right\|>t_{3}\right) \\
\label{Psi0}
& \leqslant& P_{0}\left(T_{H}\left\|\tilde{\Psibf}_{H}-\tilde{\Psibf}_{0}\right\|>\frac{t_{3}}{2}\right)+P_{0}\left(\left\|\tilde{\Psibf}_{0}\right\|\left|T_{H}-T_{0}\right|>\frac{t_{3}}{2}\right).
\end{eqnarray}

\noindent
Since $\left\|\tilde{\Psibf}_{0}\right\| = \frac{\operatorname{tr}\left(\Sigmabf_{0}\right)}{p}\left\|\Psibf_{0}\right\| \leqslant \frac{1}{\tau_{1}{ }^{2}}$, the second term in (\ref{Psi0}) satisfies 
\begin{eqnarray}\label{Psi1}
P_{0}\left(\left\|\tilde{\Psibf}_{0}\right\|\left|T_{H}-T_{0}\right|>\frac{t_{3}}{2}\right) \leqslant P_{0}\left(\left|T_{H}-T_{0}\right|>\frac{t_{3} \tau_{1}^{2}}{2}\right).
\end{eqnarray}

\noindent
Using the consistency of $\hat{\Sigmabf}_{H}$ for $\Sigmabf_{0}$, combined with Assumption H1, we get 
\begin{eqnarray}
\label{Tbound1}
0<\frac{\tau_{1}}{2}<\frac{1}{2 T_{0}}<\frac{1}{T_{H}}<\frac{2}{T_{0}}<\frac{2}{\tau_{1}}<\infty ,  
\end{eqnarray}
 and thus 
 \begin{eqnarray}
 \label{Tbound2}
 \left|T_{H}-T_{0}\right|=\left|T_{H}\right|\left|T_{0}\right|\left|\frac{1}{T_{H}}-\frac{1}{T_{0}}\right| \leqslant \frac{2}{\tau_{1}} \cdot \frac{1}{\tau_{1}} \cdot\left|\frac{1}{T_{H}}-\frac{1}{T_{0}}\right|       
 \end{eqnarray}

 \noindent
 on an event $\tilde{C}_n$ such that $P_0 (\tilde{C}_n^c) \rightarrow 0$ as $n \rightarrow \infty$. Then by combining (\ref{Psi1}) and (\ref{Tbound2}) we can get the upper bound for the second term in (\ref{Psi0}) as follows: 
 {\small
\begin{eqnarray}
\notag
P_{0}\left(\left\|\tilde{\Psibf}_{0}\right\|\left|T_{H}-T_{0}\right|>\frac{t_{3}}{2}\right)&\leqslant& P_{0}\left(\left|\frac{1}{T_{H}}-\frac{1}{T_{0}}\right|>\frac{t_{3} \tau_{1}^{4}}{4}\right) + P_0 (\tilde{C}_n^c)\\
\notag
&=&P_{0}\left(\frac{1}{p}\left|\operatorname{tr}(\hat{\Sigmabf}_{H}-\Sigmabf_{0})\right|>\frac{t_{3} \tau_{1}^{4}}{4}\right) + P_0 (\tilde{C}_n^c)\\
\label{Psihat1eq1}
&\leqslant& P_{0}\left(\left\|\hat{\Sigmabf}_{H}-\Sigmabf_{0}\right\|>\frac{t_{3} \tau_1^{4}}{4}\right) + P_0 (\tilde{C}_n^c)\\
\notag
&\leqslant&  2\exp \left(-\frac{\tilde{C}_{0}^2 c \tau_{1}^{4}}{4} \max (p, \log n)+\log 21 \cdot p\right) + P_0 (\tilde{C}_n^c) \\
\label{Psihat1}
\end{eqnarray}}

\noindent
for any $\tilde{C}_{0}>\frac{2}{\tau_1^2}\sqrt{\frac{\log 21}{c}}$ and $t_{3}\geqslant \frac{4\tilde{C}_{0}}{\tau_1^4} \sqrt{\frac{\max(p,\log n)}{nq}}$ , where the inequality (\ref{Psihat1eq1}) follows from the properties of eigenvalues, and the inequality (\ref{Psihat1}) follows from (\ref{sigmacovrate}). Also, the first term in (\ref{Psi0}) satisfies
{\small
\begin{eqnarray}
\notag
P_{0}\left(T_{H}\left\|\tilde{\Psibf}_{H}-\tilde{\Psibf}_{0}\right\|>\frac{t_{3}}{2}\right) &\leqslant&
P_{0}\left(\left\|\tilde{\Psibf}_{H}-\tilde{\Psibf}_{0}\right\|>\frac{t_{3} \tau_{1}}{4}\right) + P_0 (\tilde{C}_n^c)\\
\notag
&\leqslant&  2\exp \left(-\frac{\tilde{C}_{0}^2 c \tau_{1}^{4}}{4} \max (q, \log n)+\log 21 \cdot q\right) + P_0 (\tilde{C}_n^c)\\
\label{Psihat2}
\end{eqnarray}}
for any $\tilde{C}_{0}>\frac{2}{\tau_1^2}\sqrt{\frac{\log 21}{c}}$ and any $t_{3}\geqslant \frac{4\tilde{C}_{0}}{\tau_1} \sqrt{\frac{\max(q,\log n)}{np}}$  , where the inequality (\ref{Psihat2}) follows from (\ref{tildepsicovrate}). Combining equations (\ref{Psi0}), (\ref{Psihat1}) and (\ref{Psihat2}), for any $\tilde{C}_{0}>\frac{2}{\tau_1^2}\sqrt{\frac{\log 21}{c}}$, we get 
\begin{eqnarray}
 \notag
&&P_{0}\left(||\hat{\Psibf}_{H}-\Psibf_{0}||>\frac{4\tilde{C}_{0}}{\tau_1^4} \sqrt{\max\left(\frac{\max(p,\log n)}{nq},\frac{\max(q,\log n)}{np}\right)}\right)\\
\notag
&\leqslant& 2\exp \left(-\frac{\tilde{C}_{0}^2 c \tau_{1}^{4}}{4} \max (p, \log n)+\log 21 \cdot p\right)\\
\label{Psiteq}
& & +2\exp \left(-\frac{\tilde{C}_{0}^2 c \tau_{1}^{4}}{4} \max (q, \log n)+\log 21 \cdot q\right)+ 2 P_0 (\tilde{C}_n^c)
 \end{eqnarray}

\noindent
By the construction of $\tilde{C}_n$ and Assumption H2, it follows that all the three terms on the right side of the above inequality converge to zero as $n \rightarrow \infty$. Thus, the required spectral norm convergence rate for $\hat{\Psibf}_H$ has been established. 

Note that 
\begin{eqnarray*}
\|\hat{\Omegabf}_H-\Omegabf_0\| 
&=& \|\hat{\Sigmabf}_H^{-1}-\Sigmabf_0^{-1}\|\\
&=& \|\hat{\Sigmabf}_H^{-1} (\hat{\Sigmabf}_H-\Sigmabf_0) 
\Sigmabf_0^{-1}\|\\
&\leqslant& \|\hat{\Sigmabf}_H^{-1}\| \|(\hat{\Sigmabf}_H-\Sigmabf_0)\|
\|\Sigmabf_0^{-1}\|\\
&\leqslant& \frac{\|\hat{\Sigmabf}_H^{-1}\|}{\tau_1} \|(\hat{\Sigmabf}_H-\Sigmabf_0)\|. 
\end{eqnarray*}

\noindent
The last inequality follows by Assumption H1. By (\ref{sigmacovrate}) and Assumption H1, it follows that $\|\hat{\Sigmabf}_H^{-1}\| \leqslant 2/\tau_1$ on an event with $P_0$-probability converging to $1$ as $n \rightarrow \infty$. The required spectral norm convergence rate for $\hat{\Omegabf}_H$ now follows from (\ref{sigmacovrate}). Finally, the required convergence rate for $\hat{\Gammabf}_H$ follows by leveraging (\ref{Psiteq}) and then using similar arguments as above. 
 
\end{appendix}


\end{document}